\documentclass[12pt,twoside,letterpaper]{amsart}
\usepackage{amssymb,amsfonts,amsxtra,mathrsfs,
  placeins,graphicx,verbatim,
  fullpage}

\usepackage{url}
\usepackage{graphicx}
\usepackage{amsthm}

\numberwithin{equation}{section}
\theoremstyle{definition}
\newtheorem{theorem}[equation]{Theorem}
\newtheorem{lemma}[equation]{Lemma}
\newtheorem{proposition}[equation]{Proposition}
\newtheorem{corollary}[equation]{Corollary}

\newtheorem{definition}[equation]{Definition}

\newtheorem{example}[equation]{Example}

\newtheorem{remark}[equation]{Remark}

\newtheorem{claim}[equation]{Claim}
\renewcommand{\SS}{\mathbb{S}}

\newcommand{\pt}{\mathrm{pt}}

\newcommand{\ad}{\operatorname{ad}}

\newcommand{\Id}{\mathrm{Id}}

\newcommand{\BV}{\mathscr{BV}}

\newcommand{\ZZ}{\mathbb{Z}}

\newcommand{\iso}{{\;\stackrel{_\sim}{\to}\;}}
\newcommand{\CC}{\mathbb{C}}

\newcommand{\cO}{\mathcal{O}}

\newcommand{\g}{\mathfrak{g}}

\newcommand{\Hom}{\text{Hom}}
\newcommand{\Det}{\text{Det}}

\newcommand{\gr}{\operatorname{\mathsf{gr}}}
\newcommand{\Aut}{\operatorname{Aut}}
\newcommand{\Der}{\operatorname{Der}}
\newcommand{\CDer}{\operatorname{CDer}}
\newcommand{\End}{\operatorname{End}}

\newcommand{\onto}{\twoheadrightarrow}
\newcommand{\into}{\hookrightarrow}
\newcommand{\Sym}{\operatorname{Sym}}
\newcommand{\CE}{\operatorname{CE}}
\newcommand{\HCE}{\operatorname{HCE}}

\begin{document}
\title{Curved infinity-algebras and their characteristic classes}
\author{Andrey Lazarev and Travis Schedler}
\date{}
\begin{abstract}
  In this paper we study a natural extension of Kontsevich's
  characteristic class construction for $A_\infty$- and
  $L_\infty$-algebras to the case of curved algebras. These define
  homology classes on a variant of his graph homology which allows
  vertices of valence $\geq 1$.  We compute this graph homology, which
  is governed by star-shaped graphs with odd-valence vertices. We also
  classify nontrivially curved cyclic $A_\infty$- and $L_\infty$-
  algebras over a field up to gauge equivalence, and show that these
  are essentially reduced to algebras of dimension at most two with
  only even-ary operations.  We apply the reasoning to compute
  stability maps for the homology of Lie algebras of formal vector
  fields.  Finally, we explain a generalization of these results to
  other types of algebras, using the language of operads. A key
  observation is that operads governing curved algebras are closely
  related to the Koszul dual of those governing unital algebras.
\end{abstract}
%\subjclass[2000]{14H10, 17B56, 17B66, 81T18, 81T40.}
\thanks{ We
  would like to thank B. Vallette for helpful comments on an earlier
  version of this paper. The first author would like to thank the MPIM
  Bonn and the second author the University of Leicester for
  hospitality during some of this work. The first author was partially
  supported by an EPSRC research grant. The second author is a
  five-year fellow of the American Institute of Mathematics, and was
  partially supported by the ARRA-funded NSF grant DMS-0900233.}
\maketitle
\tableofcontents
\section{Introduction}
Kontsevich has associated certain characteristic classes to
finite-dimensional $L_\infty$- or $A_\infty$-algebras equipped with an
invariant inner product in \cite{Kncsg, KFd}. These are expressed in
terms of the homology of certain complexes spanned by graphs with some
additional structures. This construction is by now well-understood
both from the point of view of Lie algebra homology and topological
conformal field theory; see, for example, \cite{HaLacc}.

In this note, we explore a natural generalization of this construction
to the case of \emph{curved} algebras, introduced by Positselski in
\cite{Pnqdc}. It turns out that a complete description of these
classes, and of the homology of the associated graph complexes, is
possible. We show that these are all obtainable from
one-dimensional algebras, and that these classes are zero for algebras
with zero curvature. This contrasts with the corresponding problem for
conventional graph complexes, which is still widely open.

Roughly, this holds for the following reason. The usual (uncurved) $L_\infty$-algebras correspond to odd
vector fields $X$ on the formal neighborhood of zero in a vector space
satisfying $[X,X] = 0$, which have a critical point at the origin.
Equivalence classes of these are complicated.  On the other hand,
nontrivially curved $L_\infty$-algebras correspond to the same vector
fields except having no critical points, and these can all be put in a
standard form.  A similar result extends to the noncommutative setting
for $A_\infty$-algebras.  Cyclic algebras replace vector fields with
Hamiltonian vector fields, in the context of the formal neighborhood
of zero in a symplectic vector space.  In this case, the standard form
is somewhat less trivial. Our first results, Theorems \ref{curvtrivthm}
and \ref{cyccurvtrivthm}, classify all of these vector fields (i.e.,
types of algebras) up to gauge equivalence.  In particular, in the
ordinary (noncyclic) case, nontrivially curved $A_\infty$- or
$L_\infty$-algebras are gauge equivalent (i.e., homotopy isomorphic)
to the algebra with the same curvature and zero multiplication; in the
cyclic case, algebras are gauge equivalent to the direct sum of a
curved algebra of dimension at most two of a certain form (but having
nontrivial multiplications in general), with an algebra having the trivial infinity-structure.

As explained by Kontsevich, the aforementioned graph complexes can be
viewed as computing the stable homology of Lie algebras of symplectic
vector fields on a vector space $W$ (in the $A_\infty$ case, one
should take \emph{noncommutative} symplectic vector fields).  This
motivates us to consider the stability maps.  In this direction, we
prove that the map from the Lie algebra of symplectic vector fields on
$W$ vanishing at the origin to the homology of the Lie algebra of
\emph{all} vector fields on $W \oplus \CC \cdot w$, where $w$ is an
odd vector, is zero.  Similarly, we show the same for the Lie algebra
of noncommutative symplectic vector fields.

The precise relation to the previous result is as follows. Any cyclic
$L_\infty$-algebra structure on $V$ defines an unstable characteristic
class in the homology of the Lie algebra of symplectic vector fields
on the shifted vector space $W = \Pi V$. As $\dim V \rightarrow
\infty$, the homology of this Lie algebra converges to the graph
homology (at least if $V$ is only growing in the purely even or purely
odd direction: see Theorem \ref{GFgraph} below), and the image of the
unstable characteristic class under the stability maps gives, in the
limit, the aforementioned (stable) characteristic class.  Hence, our
result above says that the unstable curved characteristic class of an
algebra with zero curvature already maps to zero under the first
stability map $W:=\Pi V \into W \oplus \CC\cdot w$.

These results hint at a triviality of curved infinity-algebras
from a homological point of view, at least when the cyclic structure
is not considered. A similar result on the triviality of the
corresponding derived categories was obtained recently in
\cite{KLNnv}; another manifestation of this triviality principle is
briefly discussed in the last section of this paper.

Finally, we generalize these results to the operadic setting, i.e., to
types of algebras other than associative and Lie algebras. A key point
is that operads governing curved algebras are closely related to the
Koszul dual of operads governing unital algebras. Namely, a curved
version of an operad is obtained by the cobar construction from a
unital version of the Koszul dual operad. The presence of the unit
provides a contracting homotopy on large summands of the obtained
graph complexes.

In particular, the operadic generalization applies to Poisson,
Gerstenhaber, BV, permutation, and pre-Lie algebras. For the most
part, the generalization is straightforward, and we restrict ourselves
with giving only an outline of arguments in this section.  There is,
however, one important aspect which is less visible in the special
cases of commutative and ribbon graphs: a curved graph complex
associated with a cyclic (or even modular) operad $\mathcal O$ is
quasi-isomorphic to a variant of the \emph{deformation complex of a
  curved $\mathcal O$-algebra on a one-dimensional space}.  Therefore,
this graph complex supports the structure of a differential graded Lie
algebra. This differential Lie algebra, and its Chevalley-Eilenberg
complex, appeared in various guises in the works of Zwiebach-Sen,
Costello and Harrelson-Voronov-Zuniga on quantum master equation,
\cite{SZbias, Cospftft, HVZocms}.
\subsection{Notation and conventions}\label{notsec}
In this paper we work in the category of $\mathbb Z/2$-graded vector
spaces (also known as (super)vector spaces) over $\CC$ although all
results continue to hold in the $\mathbb Z$-graded context and when
$\CC$ is replaced by any field of characteristic zero. We will usually
refer to these graded vector spaces simply as ``spaces.'' The parity of a homogeneous vector $v$ in a space will be denoted by $|v|$. The adjective
`differential graded' will mean `differential $\mathbb Z/2$-graded'
and will be abbreviated as `dg'. A (commutative) differential graded
(Lie) algebra will be abbreviated as
(c)dg(l)a.
All of
our unmarked tensors are understood to be taken over $\CC$. For a $\mathbb Z/2$-graded vector space $V=V_0\oplus V_1$ the symbol $\Pi
 V$ will denote the \emph{parity reversion} of $V$; thus $(\Pi
 V)_0=V_1$ while $(\Pi V)_1=V_0$.

 We will make use of the language of \emph{formal}\footnote{Here the
   word ``formal'' is understood in the sense of a formal
   neighborhood, which differs from the notion of ``formality'' in
   rational homotopy theory.}  spaces and algebras (which exist since
 \cite{Lefat} under the name "linearly compact"; see, e.g.,
 \cite{HaLactha} for a recent treatment relevant to the present work,
 under the present name).  A formal space is an inverse limit of
 finite-dimensional spaces.  % If $V$ is a discrete
 % space and $W=\lim_\leftarrow{W_i}$ is a formal space we will write
 % $V\otimes W$ for $\lim_{\leftarrow}V\otimes W_i$; thus for two
 % discrete spaces $V$ and $U$ we have $\Hom(V,U)\cong U\otimes
 % V^*$.
 The functor of taking the linear dual establishes an
 anti-equivalence between the category of (discrete) vector spaces and that of
 formal vector spaces.

 It will always be clear from the context whether
 we work with formal or discrete vector spaces, and we will typically
 not mention this specifically later on; the tensor product of two formal spaces is understood to be their \emph{completed} tensor product.  Furthermore, the symbol $V$
 will be reserved for a discrete space, with its dual $V^*$ therefore
 a formal space.

  In particular, we will work with formal (c)dg(l)as.  The main
 examples will be completed tensor and symmetric algebras on formal
 spaces $W$; these will be denoted by $\hat{T}W$ and
 $\hat{S}W$ respectively. Note that we will \emph{never} consider the
 uncompleted $TW$ and $SW$ when $W$ is formal, and similarly never
 consider the completed $\hat{T}V$ or $\hat{S}V$ when $V$ is discrete,
 so as to stay in either the category of formal spaces or that of
 discrete spaces.

 The Lie algebras of continuous derivations of $\hat{T}W$ and
 of $\hat{S}W$ will be denoted by $\Der(\hat{T}W)$ and
 $\Der(\hat{S}W)$ respectively; we will also consider their Lie
 subalgebras $\Der^0(\hat{T}W)$ and $\Der^0(\hat{S}W)$ consisting of
 derivations having no constant terms.

We mention here two potential pitfalls present in this framework. Firstly,
the categories of discrete and formal vector spaces are not disjoint: the
spaces which are both discrete and formal are precisely finite-dimensional
spaces. And secondly, not every space is either discrete or formal; moreover
such spaces arise as a result of some natural operations with discrete or
formal spaces. For example if $U$ and $W$ are infinite-dimensional formal or
discrete spaces then the vector space $\Hom(U,W)$ is neither formal nor
discrete. Similarly, the Lie algebras
$\Der(\hat{T}V^*)$ and $\Der({T}V)$ will be neither
formal nor discrete if $V$ is infinite-dimensional.

Therefore, to avoid possible confusion, we make the blanket assumption
that the $\ZZ/2$-graded vector space $V$ which appears throughout the
paper, in addition to being discrete as above, is in fact
\emph{finite-dimensional}.  This way all the objects we consider will
live in either the category of formal spaces or the category of
discrete spaces (but not both: so each finite-dimensional space we
consider will be viewed in only one way). The price we pay is that
some of our results are not formulated in maximal generality; namely
Theorem \ref{curvtrivthm} and Claim \ref{opclaim} do not need the
space $V$ to be finite-dimensional (although essentially none of the
exposition needs to be modified to obtain this generalization).

 For a \emph{formal} dgla $\g$ its Chevalley-Eilenberg cohomological
 complex will be denoted by $\CE^\bullet(\g)$.  This is defined as
 follows (note that the definition \emph{differs} from the usual one
 in where completions are taken, since $\g$ is formal rather than
 discrete):
the underlying graded vector space of $\CE^\bullet(\g)$ is $S\Pi{\mathfrak g}^*$ and the differential is given as a sum of two maps $d_{\operatorname I}$ and $d_{\CE}$. Here $d_{\operatorname I}$ and $d_{\CE}$ are both specified by their restriction onto $\Pi{\mathfrak g}^*$ and extended to the whole $S\Pi{\mathfrak g}^*$ by the Leibniz rule; further $d_{\operatorname I}:\Pi{\mathfrak g}^*\to \Pi{\mathfrak g}^*$ is the shift of the dual of the internal differential on ${\mathfrak g}$ whereas $d_{\CE}:\Pi{\mathfrak g}^*\to S^2(\Pi{\mathfrak g}^*)$ is induced by the commutator map $[,]:{\mathfrak g}\otimes {\mathfrak g} \to {\mathfrak g}$.

This is in general a $\ZZ/2$-graded complex.  In the case that the
differential $d$ is zero, it has an additional grading by
\emph{cohomological degree}, i.e., $S^i \Pi\mathfrak{ g}^*$ is in degree
$i$.  The corresponding homological complex is the linear dual:
$\CE_\bullet(\g)=(\CE^\bullet(\g))^*$, which has the underlying formal space $\hat S \Pi \mathfrak{g}$.
Note that in some papers (e.g. \cite{HaLacc}) the Chevalley-Eilenberg complex of a graded Lie algebra ${\mathfrak g}$ is defined using the (in our
case completed) \emph{exterior} algebra $\hat \Lambda {\mathfrak g}$; this definition is equivalent to ours under a canonical isomorphism $\hat S(\Pi{\mathfrak g})\cong \hat \Lambda {\mathfrak g}$ where
\[\Pi g_1\ldots\Pi g_n\mapsto (-1)^{|g_{n}|+2|g_{n-1}|+\ldots+(n-1)|g_1|}g_1\wedge\ldots\wedge g_n.\]

 \section{Gauge equivalence classes of curved (cyclic) $A_\infty$- and
   $L_\infty$-algebras}
\subsection{Curved $A_\infty$- and $L_\infty$- algebras}
We recall the definition of $A_\infty$- and $L_\infty$-algebras
following \cite{HaLacc}, as well as their curved analogues; cf.~e.g.,
\cite{Nicbd} and references therein.

A curved $A_\infty$-algebra structure on a
(finite-dimensional\footnote{As in \S \ref{notsec}, $V$ is considered
  as a discrete space throughout; in the present subsection one could
  allow it to be infinite-dimensional and discrete.  We will not make
  further mention of this.}) space $V$ is a continuous odd derivation
$m$ of the formal dga $\hat{T}\Pi V^*$ and a curved $L_\infty$-algebra
structure on $V$ is a continuous odd derivation $m$ of the formal cdga
$\hat{S}\Pi V^*$; additionally $m$ is required to square to zero in
both cases. An ordinary (i.e. uncurved) $A_\infty$- or
$L_\infty$-structure is specified by the requirement that $m$ have no
constant term. The components $m_i:{T}^i\Pi V\to \Pi V$ or ${S}^i\Pi
V\to \Pi V$ of the dual of $m$ are the structure maps of the
corresponding $A_\infty$- or $L_\infty$-structure.

We note that sometimes it is convenient to use the more traditional
way of writing the structure maps of an $A_\infty$- or
$L_\infty$-algebra $V$ as ${T}^iV\to V$ or ${\Lambda}^iV\to V$; these
maps will then be even or odd depending on whether $i$ is even or
odd. To alleviate the notation we will still write $m_i$ for these
maps when the meaning is clear from the context.

We are interested in \emph{gauge equivalence} classes of $A_\infty$- or $L_\infty$- structures on a fixed  space $V$. In particular, this equivalence relation implies other relations found in the literature under the names of homotopy or quasi-isomorphism.

Namely, a gauge equivalence between $(V,m)$ and $(V,m')$ is defined as
a derivation $\xi \in \Der^0(\hat{T}\Pi V^*)$ or $\xi \in
\Der^0(\hat{S} \Pi V^*)$ which is even and satisfies $m' = e^{\ad \xi}
m$.\footnote{In the case that our ground field is not $\CC$, $e^\xi$
  still makes sense if we require additionally that $\xi_1 = 0$ as
  well, i.e., $\xi$ has no linear term.  We can then modify the above
  by saying that a gauge equivalence is a composition of such an
  equivalence (for $\xi_1=0$) with a linear isomorphism of $V$.} In
particular, such a gauge equivalence yields an isomorphism of dgas
$e^\xi: (\hat{T} \Pi V^*,m) \iso (\hat{T} \Pi V^*, m')$ or cdgas
$e^\xi: (\hat{S} \Pi V^*,m) \iso (\hat{S} \Pi V^*, m')$, and we usually
denote the gauge equivalence by $e^\xi$.
\begin{theorem}\label{curvtrivthm}
  If $(V, m)$ is a curved $A_\infty$- or $L_\infty$- algebra for which
  the curvature $m_0 = c \in V_0$ is nonzero, then $m$ is gauge
  equivalent to the structure $m' = c$ with all higher multiplications
  $m'_i = 0$ for $i > 0$.
\end{theorem}
Roughly, the above is saying that, when an (odd, noncommutative)
formal vector field is nonzero evaluated at zero, then it is
equivalent to a constant vector field up to (generally nonlinear) change of coordinates.

As a corollary of the theorem, it follows that any two
nontrivially curved algebras with the same underlying graded vector space $V$
are gauge equivalent.
\begin{proof}[Proof of Theorem \ref{curvtrivthm}]
  We consider the $A_\infty$ case; the $L_\infty$ case is similar.

  Any $A_\infty$-algebra structure $(V,m')$ with $m'_0=m_0=c$ can be
  viewed as a deformation of $(V,m_0)$. Indeed, let us introduce a
  formal parameter $\hbar$; then $(V,m')$ is equivalent to the
  deformed structure $(V,m'_\hbar)$, where $m'_\hbar = m'_0 + \sum_{i
    \geq 1} \hbar^i m'_i$.  This yields an equivalence with deformed
  structures whose $i$-ary operations are homogeneous of degree $i$ in
  $\hbar$.

  The structures of the form $(V,m')$ are governed by the dgla
  $\Der(\hat T \Pi V^*, [,c])$, where $c \in V$ is viewed as an odd
  constant derivation of $\hat T \Pi V^*$.  Formal deformations (such
  as $(V, m'_\hbar)$) are governed by the dgla $\Der(\hat T \Pi V^*[[\hbar]],
  [,c])$.  Gauge equivalences $e^{\xi}$ of Maurer-Cartan elements of
  $\Der(\hat T \Pi V^*, [,c])$, which we can assume satisfy $\xi_0 = 0
  = \xi_1$ (so as to not change the curvature) are identified with
  gauge equivalences $e^{\frac{1}{\hbar}\xi_\hbar}$ (for $\xi_\hbar :=
  \sum_{i} \hbar^i \xi_i$) of the corresponding Maurer-Cartan elements
  of $\Der(\hat T \Pi V^*[[\hbar]], [,c])$, which preserve the grading
  $|\hbar|=1=|V^*|$.

 Since gauge equivalence classes of deformations of Maurer-Cartan
 elements are preserved under quasi-isomorphisms of dglas, it suffices
 to show that $\Der(\hat T \Pi V^*, [,c])$ is acyclic. Let us write
 down this complex $C^\bullet$ explicitly; note that $C^\bullet$ is a version of
 the Hochschild complex in the curved setting.

 Set $C^i:=\Hom((\Pi V)^{\otimes i}, \Pi V)$ with the differential $d:C^i\to
 C^{i-1}$  given by the formula, for $f\in
 C^i$, and $x_1, \ldots, x_{i-1} \in \Pi V$:
\[
df(x_1,\ldots, x_{i-1})=\sum_k(-1)^{|x_1|+\cdots+|x_k|}f(x_1,\ldots, x_{k},c,x_{k+1},\ldots, x_{i-1}).
\]
We construct an explicit contracting homotopy.  Choose an odd linear
map $\epsilon:\Pi V\to\CC$ such that $\epsilon(c)=1$ (here, odd means
that $\epsilon|_{\Pi (V_1) = (\Pi V)_0} = 0$). Define maps $s_i:C^i\to
C^{i+1}$ by the formula, for $f\in C^i$:
\begin{equation} \label{siref}
s_if(x_1,\ldots,x_{i+1})=\epsilon(x_1)f(x_2,\ldots, x_{i+1}).
\end{equation}
Then,
\begin{equation}
d s_i f(x_1,\ldots,x_{i}) + s_{i-1} d f(x_1, \ldots, x_{i}) = \epsilon(c) f(x_1, \ldots, x_{i}) = f(x_1, \ldots, x_{i}). \qedhere
\end{equation}
\end{proof}
\subsection{Cyclic algebras}
We now extend the results of the previous subsection to the case of
algebras with a cyclic inner product.  By an \emph{inner product} on a
$\ZZ/2$-graded vector space $V$ we mean a nondegenerate symmetric
graded bilinear form $( -, - ): V\otimes V \rightarrow \CC$, where
$\CC$ is considered to be in even degree, and $V$ is required to be
finite-dimensional. An \emph{inner product space} is a space equipped
with an inner product.
\begin{definition} Let $(V,m)$ be a finite-dimensional $A_\infty$- or
  $L_\infty$-algebra.  A cyclic inner product on $V$ is an inner
  product $( -, - ): V\otimes V \rightarrow \CC$ for which the tensors
  $( m_i(v_1,\ldots,v_i), v_{i+1})$ are invariant with respect to the
  signed cyclic permutations of arguments:
\[
(
  m_i(v_1,\ldots,v_i), v_{i+1})=
 (-1)^{i+(|v_1|+\ldots+|v_{i}|)|v_{i+1}|}( m_i(v_{i+1},v_1\ldots,v_{i-1}), v_{i})
 \] We call an algebra $(V,m)$ a \emph{cyclic} $A_\infty$- or
 $L_\infty$-algebra when it is equipped with such a cyclic inner
 product.
\end{definition}
We need to define a subspace of $\Der(\hat{T}\Pi V^*)$ of
\emph{cyclic} derivations.  To do so, first note that $\Der(\hat{T}
\Pi V^*) \cong \Hom(\Pi V^*, \hat{T}\Pi V^*)$ via the restriction map.
Using the inner product, we can identify the latter with
$(\Pi V^* \otimes
\hat{T}\Pi V^*)$.  Next, on any tensor power $(\Pi V^*)^{\otimes n}$, we
can define the (graded) cyclic permutation operator $\sigma:
(\Pi V^*)^{\otimes n} \rightarrow (\Pi V^*)^{\otimes n}$. This extends to a
continuous linear automorphism of $(\Pi V^* \otimes \hat{T}\Pi V^*)$, and
yields a continuous linear automorphism of $\Der(\hat{T} \Pi
V^*)$.
\begin{definition}\label{cycderdefn}
  Let $\CDer(\hat{T} \Pi V^*)$ denote the space of derivations
  which are invariant under cyclic permutation: we call these
  \emph{cyclic} derivations.  Similarly, define $\CDer^0(\hat{T} \Pi V^*)$ as those cyclic derivations with zero constant term (i.e., preserving the
augmentation $\hat{T}^{\geq 1} \Pi V^*$). Finally, define the spaces $\CDer(\hat{S} \Pi V^*)$ and $\CDer^0(\hat{S} \Pi V^*)$ analogously.
\end{definition}
Note that $\CDer(\hat{S} \Pi V^*)$ can be
viewed as the subspace of $\CDer(\hat{T} \Pi V^*)$ landing in symmetric
tensors, and similarly for the version with zero constant term.

Given $V$ together with an inner product, cyclic curved
$A_\infty$-structures are the same as odd derivations $\xi \in
\CDer(\hat{T} \Pi V^*)$ which square to zero.  Similarly,
uncurved structures correspond to odd square-zero $\xi \in
\CDer^0(\hat{T} \Pi V^*)$, and the $L_\infty$-versions are
obtained by replacing $\hat{T}$ with $\hat{S}$.  See, e.g.,
\cite{HaLactha} for details.
\begin{definition} A gauge equivalence of cyclic $A_\infty$- or
  $L_\infty$- structures $(V,m)$ and $(V,m')$ on a fixed inner product
  space $V$ is a map $e^\xi$ where $\xi \in \CDer^0(\hat{T}\Pi V^*)$
  or $\xi \in \CDer^0(\hat{S} \Pi V^*)$ satisfies $m' = e^{\ad \xi}
  m$.
\end{definition}
\begin{remark}
  In the literature, the term \emph{symplectic} is sometimes used
  instead of cyclic, the idea being that an inner product
  on $V$ is equivalent to a (constant) symplectic structure on $\Pi
  V$, so that the cyclic derivations on $\hat S \Pi V^*$ or $\hat T \Pi V^*$ are the same
  as formal (possibly noncommutative) symplectic vector fields on $\Pi V$; a cyclic gauge equivalence could then be interpreted as a (formal, noncommutative) symplectomorphism (preserving the $\ZZ/2$-grading).
Note that it follows from this interpretation that
the cyclic derivations of $V$ form a Lie superalgebra, and the cyclic
gauge equivalences a Lie group.
\end{remark}
\begin{theorem}\label{cyccurvtrivthm}\
 \begin{enumerate}
\item[(a)] If $(V, m)$ is a curved $A_\infty$-algebra with a
  cyclic inner product for which the curvature $m_0 = c \in V_0$ is
  nonzero, and $c' \in V$ any even element for which $(c, c') = 1$, then $(V,m)$ is gauge equivalent to the
  structure $m'$ with $m'_0 =c$, and higher multiplications
\begin{gather}
  m_{2i-1}' = 0, i \geq 1, \\
  m_{2i}'(x_1, \ldots, x_{2i}) = \bigl(\prod_{j=1}^{2i} (c', x_j) \bigr) \cdot (m_{2i}(c,c,\ldots,c), c) \cdot c'.
\end{gather}
\item[(b)] If $(V,m)$ is a curved $L_\infty$-algebra with a cyclic inner product for which the curvature $m_0 = c \in V_0$ is nonzero, then $(V,m)$ is gauge equivalent to the
  structure $m'$ with $m'_0 =c$, and higher multiplications $m'_i = 0$ zero for all $i \geq 1$.
\end{enumerate}
\end{theorem}
One immediately deduces
\begin{corollary}\
\begin{enumerate}
\item[(a)]
Two curved cyclic $A_\infty$-algebra structures $(V,m)$ and $(V,m')$ on the same underlying inner product space $V$ with nonzero curvature $m_0, m_0'$
are gauge equivalent if and only if
\begin{equation}
(m_0, m_0) = (m_0', m_0'), \quad (m_{2i}(m_0, \ldots, m_0), m_0) = (m_{2i}'(m_0', \ldots, m_0'), m_0'), \forall i \geq 1.
\end{equation}
\item[(b)] Two curved cyclic $L_\infty$-algebra structures $(V,m)$ and
  $(V,m')$ on the same underlying inner product space $V$ with
  nonzero curvature $m_0, m_0'$ are gauge equivalent if and only if
  $(m_0, m_0) = (m_0', m_0')$.
\end{enumerate}
\end{corollary}
\begin{remark}
Another, perhaps intuitively more clear, way to understand this result is as follows. Consider the curved cyclic $A_\infty$-algebra $V$ with curvature $c$ whose underlying inner product space is spanned by a single even vector $c$ with $(c,c)=1$ and higher products $m_{2i}(c,\ldots,c)=t_ic$; $m_{2i+1}=0$ for $i=0,1,\ldots$. Here $t_i$ are arbitrary numbers.

Consider also  the curved cyclic $A_\infty$-algebra $V^\prime$ whose underlying space is spanned by two even vectors $c$ and $c^\prime$ with $(c,c^\prime)=(c^\prime,c)=1, (c,c)=(c^\prime,c^\prime)=0$. The $A_\infty$-structure is given as
$m_0 = c$ and, for $i > 0$, $m_{2i}(c,\ldots, c)=t_ic^\prime$, where $t_i$ are arbitrary, and all other higher products are zero.

Then any cyclic curved $A_\infty$-algebra with nonzero curvature is gauge equivalent to the direct sum of either $V$ or $V^\prime$ with an $A_\infty$-algebra having zero $A_\infty$-structure.
\end{remark}
\begin{proof}[Proof of Theorem \ref{cyccurvtrivthm}]
 (a) We modify the previous
  obstruction theory argument. Deformations of the algebra with $m_0=c$ and all
higher operations zero are governed by the subcomplex
\begin{equation}\label{bbeq}
B^\bullet := \CDer(\hat{T}\Pi V^*, [,c]) \subset C^\bullet
\end{equation}
of the one
considered in Theorem \ref{curvtrivthm}.

In Lemma \ref{qil} below, we show that $B^{\bullet}$ is quasi-isomorphic to the subcomplex
$B^{\bullet}_0$ spanned by
 the cocycles
\begin{equation}\label{ceq}
\epsilon^{i+1}(x_1, \ldots, x_i) := \epsilon(x_1) \cdots \epsilon(x_i) c', \quad \epsilon(v) := (c', v).
\end{equation}
Note that, when $i$ is odd, $\epsilon^{i+1} = 0$. Moreover, $B^{\bullet}_0$
is an abelian sub-dgla with zero differential.

Using the lemma, the result follows as in the proof of Theorem
\ref{curvtrivthm}. In more detail, the gauge equivalence classes of
Maurer-Cartan elements of $B^\bullet[[\hbar]]$ and
$B^\bullet_0[[\hbar]]$ which are zero modulo $\hbar$ are
identified. Moreover, the formal Maurer-Cartan elements which are
homogeneous of the form $\sum_{i \geq 1} \hbar^i m_i$ for $m_i \in B^i$
or $B^i_0$ are identified with actual Maurer-Cartan elements $\sum_{i
  \geq 1} m_i$ with zero constant term.  Hence, the gauge equivalence
classes of cyclic curved $A_\infty$- structures with $m_0 = c$ are
identified with gauge equivalence classes of Maurer-Cartan elements of
$B^\bullet_0$ with zero constant term. Since the latter is abelian
with zero differential, all elements are Maurer-Cartan, and define
distinct gauge equivalence classes.

We deduce that all curved $A_\infty$-structures with $m_0 = c$ are gauge
equivalent to one with all odd operations equal to zero, and even operations given by some multiple of the operation
\begin{equation}
(x_1, \ldots, x_{2i}) \mapsto
\bigl(\prod_{j=1}^{2i} (c', x_j) \bigr) \cdot c'.
\end{equation}
It remains to show that this multiple is $(m_{2i}(c,c,\ldots,c), c)$.  This follows by taking the component of the original operation in the direct summand $B^i_0 \subseteq B^i = (B^i_0 \oplus B^i_+)$, using the definition of $B^i_+$ in
\eqref{bipdfn} below. One must show that the result is gauge equivalent to
the initial structure. To do so, one can inductively construct a gauge equivalence killing off the part in $B^i_k$ below for $k \geq 1$ (first inductively on $i$, then reverse-inductively on $k$), parallel to the construction of the contracting homotopy in Lemma \ref{qil}.  We omit the details.

(b) We can apply the same argument as above, except now with
$B:=\CDer(\hat{S}\Pi V^*, [,c])$.  The same argument as above applies
and we deduce the same result, except that this time
$m_{2i}(c,c,\ldots,c) = 0$ for all $i \geq 1$ by skew-symmetry of $m_{2i}$.
\end{proof}

\begin{lemma}\label{qil} Keeping the assumptions and notation
  of the theorem, the complex $B^\bullet$ decomposes as $B^\bullet =
B^\bullet_0 \oplus B^\bullet_+$, where $B^\bullet_+$ is acyclic.  Moreover,
the inclusion $B^\bullet_0 \into B^\bullet$ is a quasi-isomorphism of dglas.
\end{lemma}
\begin{proof}
The second statement follows from the first, since $B^\bullet_0 \into B^\bullet$ is a dgla map.

To prove the first statement, we modify the contracting
  homotopy $s_i$ from Theorem \ref{curvtrivthm} to act on $B^\bullet$. The result will \emph{not}  be a contracting homotopy, but will instead accomplish the desired goal.

Define maps $s_i':B^i\to B^{i+1}$ by
\begin{equation}\label{sipfla1}
s_i' f = \sum_{j=0}^{i+1} \sigma^j (s_i f),
\end{equation}
where $\sigma^j$ is the $j$-th power of the cyclic permutation defined above Definition \ref{cycderdefn}.
We now compute $s_{i-1}' d + d s_i'$.  Assume that $f \in B^i$.  In this
case, we may use the formula
\begin{multline}\label{sipfla2}
s_i' f(x_1,\ldots,x_{i+1}) = \sum_{j=1}^{i+1} (-1)^{|x_1|+\ldots+|x_{j-1}|} \epsilon(x_j) f(x_1, \ldots, x_{j-1}, x_{j+1}, \ldots, x_{i+1})
\\
+ (-1)^{|x_1| + \ldots + |x_{i+1}|} (f(x_1, \ldots, x_i), x_{i+1}) c'.
\end{multline}
Then, we compute
\begin{multline}\label{sipfla3}
(s_{i-1}' d + d s_i')f(x_1, \ldots, x_i) =  (i+1) \epsilon(c) f(x_1, \ldots, x_i) \\ - \bigl( \sum_{j=1}^{i+1} \epsilon(x_j) f(x_1, \ldots, x_{j-1}, c, x_{j+1}, \ldots, x_{i}) + (f(x_1, \ldots, x_{i}), c)c' \bigr).
\end{multline}
Thus, the operator on the RHS is contractible.  Next, for $k \leq i+1$, define subspaces
\begin{gather}
  B^i_k \subset B^i, \quad B^i_k = B^i \cap \bigl(\CC[S_{i+1}] \cdot \bigl( (\CC \cdot \epsilon)^{\otimes (i+1-k)} \otimes ((\CC \cdot c)^\perp)^{\otimes k} \bigr) \bigr), \label{bipdfn0}\\
  B^i_+ := \sum_{k =1}^{i+1} B^i_k. \label{bipdfn}
\end{gather}
Note that $B^\bullet_+$ is a subcomplex, and $B^\bullet = B^\bullet_0
\oplus B^\bullet_+$.

We claim that the RHS of \eqref{sipfla3} acts as $k \cdot \Id$ on
$B^i_k$ for all $k$. This follows directly.  As a result, $s_{i-1}' d + d
s_i'$ restricts to zero on $B^\bullet_0$ and to an automorphism on
$B^\bullet_+$.  This proves the lemma. \qedhere

\end{proof}
\begin{remark}
  Lemma \ref{qil} is equivalent to the statement that the cyclic
  (co)homology of any curved $A_\infty$-algebra $V$ whose structure
  maps are all zero, except for $m_0$ (which is not zero), is
  isomorphic to the cyclic (co)homology of a one-dimensional
  $A_\infty$-algebra $V_1$ having the same property. See
  \cite{GJainfa, HaLactha} for the notion of cyclic (co)homology of
  $A_\infty$-algebras. These (co)homologies could be computed in
  other ways from the above, for instance, with the help of Connes'
  exact sequence connecting Hochschild (co)homology with cyclic
  (co)homology (which appears in \cite{GJainfa} in the curved
  setting), since Theorem \ref{curvtrivthm} shows that the Hochschild
  (co)homologies of $V$ and $V_1$ are trivial.
\end{remark}
\section{Characteristic classes of curved algebras}
\subsection{The uncurved case}\label{ccs}
We briefly recall Kontsevich's construction of characteristic classes
of finite-dimensional $A_\infty$- or $L_\infty$-algebras with cyclic
inner products.
First, we recall the definition of certain graph complexes, for which
cyclic $A_\infty$- or $L_\infty$-algebras will produce cycles.
\begin{definition}
  A graph is a tuple $(H,V,E, \varphi_V, \varphi_E)$ of sets $H, V, E$
  of \emph{half-edges}, \emph{vertices}, and \emph{edges}, and
  surjective maps $\varphi_V: H \rightarrow V, \varphi_E: H
  \rightarrow E$ such that the fibers of $\varphi_E$ all have
  cardinality two.
\end{definition}
Given $\Gamma = (H, V, E, \varphi_V, \varphi_E)$, we will also write $H_\Gamma = H, V_\Gamma = V, E_\Gamma = E, \varphi_V^\Gamma =
\varphi_V$, and $\varphi_E^\Gamma = \varphi_E$.
\begin{definition}
  A ribbon graph is a graph together with a cyclic ordering on each
  fiber $\varphi_V^{-1}(t)$.
\end{definition}
Intuitively, one may think of the edges of ribbon graphs as slightly
fattened, which explains the cyclic ordering at vertices.

Kontsevich's graph complexes have a basis of graphs of a certain type,
with differential taking a graph to the sum over all edges of the
contracted graph obtained by shrinking that edge to a point, together
with a sign.  To make this precise requires the notion of
\emph{orientation}:
\begin{definition}
  An orientation on a graph is a choice of ordering of all the
  half-edges, ordering of all the vertices, and a sign $\pm 1$, modulo
  the relation that applying a transposition to the ordering of either
  the half-edges or the vertices is the same as changing the sign.
\end{definition}
\begin{definition}
  An oriented graph is a graph equipped with an orientation. An
  isomorphism of oriented graphs is an isomorphism of graphs which
  preserves orientation.  Similarly, the same definition applies replacing
  ``graph'' with ``ribbon graph.''
\end{definition}
Next, given a graph $\Gamma = (H,V,E, \varphi_V, \varphi_E)$ and an
edge $e \in E$ with endpoints $v_1, v_2 \in V$ meeting halves $h_1,
h_2 \in H$, one defines the contracted graph $d_e(\Gamma) = (H, V /
\{v_1 = v_2\}, E, \varphi_V', \varphi_E)$ by identifying the endpoints
$v_1$ and $v_2$. If $\Gamma$ is moreover a ribbon graph, with the
cyclically ordered sets $\varphi_V^{-1}(v_1) = (a_1, a_2, \ldots, a_i
= h_1)$ and $\varphi_V^{-1}(v_2) = (b_1, b_2, \ldots, b_j = h_2)$,
then the cyclic ordering of the half edges at the new vertex $v = v_1
= v_2$ is defined as $(a_1, a_2, \ldots, a_{i-1}, b_1, b_2,
\ldots, b_{j-1})$. Finally, if $\Gamma$ is equipped with an
orientation, where the half-edges are ordered as $h_1, h_2, p_1,
\ldots, p_m$ and with vertices ordered by $v_1, v_2, w_1, \ldots,
w_{k}$, the new orientation is given by the ordering $p_1, \ldots,
p_m$ and $v, w_1, \ldots, w_{k}$ of vertices, without changing the
sign.

Consider the graded vector space with basis the isomorphism classes of
oriented graphs \emph{whose vertices have valence $\geq 2$}, modulo
the relation that a graph is negative its opposite orientation. The
grading is given by the number of vertices.  Let $\mathcal{G}$ be the
completion of this graded vector space with respect to the number of
edges (so \emph{not} with respect to the defining grading on the
vector space, which is by number of vertices).  Similarly, define
$\mathcal{G}_r$ using ribbon graphs rather than graphs.  In other
words (since we work over a field of characteristic zero, so that $2$
is invertible), these are the spaces of possibly infinite linear
combinations of isomorphism classes of oriented graphs which are not
isomorphic to the graph obtained by reversing the orientation.

Then, it is a result of \cite{Kncsg} that
\begin{equation}
d(\Gamma) := \sum_{e \in E_\Gamma} d_e(\Gamma)
\end{equation}
defines a differential on $\mathcal{G}$ and $\mathcal{G}_r$.
\begin{definition}
  Kontsevich's graph complex is defined as $(\mathcal{G}, d)$, and his
  ribbon graph complex is defined as $(\mathcal{G}_r, d)$.
\end{definition}
\begin{remark}
  We could alternatively have used the uncompleted graph complex
  above; however, the completed version is the one which naturally
  contains characteristic classes of $L_\infty$- or
  $A_\infty$-algebras.  In particular, taking homology commutes with
  taking completion, for the following well-known reason: We can write the
  uncompleted graph complex as a direct sum of the subcomplexes of
  graphs of a fixed genus (i.e., first Betti number of the graph as a
  topological space).  For each fixed genus, the completion with
  respect to number of edges is the same as the completion with
  respect to the grading, i.e., the number of vertices, since there
  are only finitely many graphs with a fixed genus and number of
  vertices.  Hence, the completed graph complex is the same as the
  direct product of the completions of these subcomplexes with respect
  to their usual grading.
\end{remark}
Finally, given a cyclic $A_\infty$-algebra $V$, one constructs an
element of $\mathcal{G}_r$, given by a sum
\begin{equation}
[V]:=\sum_{\Gamma} \frac{1}{|\Aut(\Gamma)|} c_\Gamma(V) \cdot \Gamma,
\end{equation}
where we sum over isomorphism classes of ribbon graphs $\Gamma$ (with
group of automorphisms $\Aut(\Gamma)$), and $c_\Gamma$ is given as
follows.  Equip $\Gamma$ with an orientation. We will define
$c_\Gamma$ so that the opposite orientation would produce $-c_\Gamma$.
Namely, $c_\Gamma$ is given by contracting the multiplications $m_i$
of $V$ with the pairings $( -, -)$ according to the graph.
In more detail, consider
\begin{equation} \label{multmaps} \prod_{i=1}^n ( m(-), -
  ): \bigotimes_{i=1}^n V^{\otimes |\phi_V^{-1}(v_i)|}
  \rightarrow \CC.
\end{equation}
Let $h_1, \ldots, h_{|H|}$ be the ordering of the half-edges and $v_1,
\ldots, v_{|V|}$ the ordering of the vertices defined by the
orientation, and assume that the sign is $1$. Let us pick
\emph{ciliations} of each of the vertices $v_1, \ldots, v_{|V|}$,
which means a linear ordering of the half-edges $\varphi^{-1}_V(v_i)$
meeting each vertex $v_i$, compatible with the cyclic ordering given
by the ribbon structure.  Up to changing the sign, let us assume that
the ordering of the half-edges is $\varphi^{-1}_V(v_1),
\varphi^{-1}_V(v_2), \ldots, \varphi^{-1}_V(v_{|V|})$.  Let $f \in V
\otimes V$ be the inverse to the pairing $(-,-): V \otimes V
\rightarrow \CC$. Finally, pick an arbitrary ordering of the edges
$e_1, \ldots, e_{|E|}$. Then, one applies \eqref{multmaps} to the
element obtained by applying the signed permutation of components of
$f^{\otimes |E|} \in V^{\otimes |H|}$ which rearranges the half-edges
$\varphi_E^{-1}(e_1), \ldots, \varphi_E^{-1}(e_{|E|})$ into the
ordering $h_1, \ldots, h_{|H|}$. One can check that the result does
not depend on the choices of orderings (but only depends on the
orientation of $\Gamma$ by a sign, as mentioned above), and we let
$c_\Gamma(V)$ to be the result of this computation.

In a similar manner, one constructs from any cyclic $L_\infty$-algebra $V$
an element $[V]$ of $\mathcal{G}$.  Then, the following result
is due to Kontsevich:
\begin{proposition}\cite{KFd} If $V$ is a cyclic $A_\infty$- or $L_\infty$-algebra then $[V]$ is a cycle on
  $(\mathcal{G}, d)$ or $(\mathcal{G}_r, d)$ respectively.
\end{proposition}
  A direct proof of the proposition in the $A_\infty$-case
  is contained, e.g., in \cite{Igu}. More conceptually, one can view a
  cyclic $A_\infty$-algebra $V$ as an algebra over ${\mathsf
    F}\underline{\mathscr{A}\textit{ss}}^{1}$, the Feynman transform of the $\Det$-twisted naive
  modular closure of the cyclic operad $\mathscr{A}\textit{ss}$; see \cite{GKmo} and
  \cite{ChLadft} concerning these notions. Therefore, we obtain a map of
  modular operads ${\mathsf F}\underline{\mathscr{A}\textit{ss}}^1\to
  \mathscr{E}(V)$, where $\mathscr{E}(V)$ is the modular endomorphism
  operad of $V$. The map between the vacuum parts of the corresponding
  operads
\[{\mathsf F}\underline{\mathscr{A}\textit{ss}}^1((0))\to
\mathscr{E}(V)((0))\cong \CC
\]
is precisely the characteristic class described above, and it follows that it does indeed give a cycle. One can prove the proposition similarly
in the $L_\infty$ case.

\subsection{Curved characteristic classes}
The preceding results have a natural generalization to the case of
curved algebras.  We need to remove the valence $\geq 2$ condition,
however, and study the graph complexes $(\widetilde{\mathcal{G}}, d),
(\widetilde{\mathcal{G}_r}, d)$ of formal linear combinations of
graphs and ribbon graphs where vertices are allowed to have valence
$1$ (we will not allow valence-zero vertices, since they don't add
anything of value). We can consider this to be the complex of ``graphs
with leaves,'' where a leaf is an edge incident to a valence-one
vertex. Note that the conventional graph complex $(\mathcal{G}, d)$ is
a subcomplex of $(\widetilde{\mathcal{G}}, d)$ and similarly in the
ribbon case. Then, everything else goes through exactly as above, and
we obtain the following result.
\begin{proposition} Any curved cyclic $A_\infty$- or
  $L_\infty$-algebra $V$ induces a cycle $[V]$ on
  $(\widetilde{\mathcal{G}_r}, d)$ or on $(\widetilde{\mathcal{G}},
  d)$ respectively.
\end{proposition}
The curved graph homology could be expressed in terms of certain
Gelfand-Fuks type homology. Namely, let $W$ be a graded symplectic
vector space and consider the graded formal Lie algebra $\g(W)$ of formal
symplectic vector fields on $W$. Similarly consider the graded formal Lie
algebra $\g_r(W)$ consisting of formal \emph{noncommutative}
symplectic vector fields on $W$, i.e. the Lie algebra
$\CDer(\Pi W)$. Taking the stable limit as the dimension
of the even or odd part of $W$ goes to infinity, we arrive at
following result (which
is a straightforward adaptation of \cite[Theorem 1.1]{Kncsg}, up to technical
problems stemming from the lack of complete reducibility of finite-dimensional representations of simple Lie superalgebras).  Let
$\CC^{m}$ denote the even space of dimension $m$ and $\Pi \CC^m$ the odd
space of dimension $m$.  We equip $\CC^{2m}$ with the standard symplectic
form, and $\Pi \CC^m$ with the standard odd symplectic (i.e., orthogonal) form.
\begin{theorem}\label{GFgraph}
  Let $W$ be a fixed inner product space.
  There are isomorphisms
\begin{equation}
\operatorname{H}_\bullet
  (\widetilde{\mathcal{G}_r}) \cong \lim_{m \rightarrow \infty}
   \HCE_\bullet(\g_r(W \oplus \CC^{2m})), \quad \operatorname{H}_\bullet
  (\widetilde{\mathcal{G}_r}) \cong \lim_{m \rightarrow \infty}
   \HCE_\bullet(\g_r(W \oplus \Pi \CC^m))
\end{equation}
 between the stable Chevalley-Eilenberg
  homology of the Lie algebra $\g_r$ and of the
  corresponding version of the curved graph complex.
Similarly, we have isomorphisms
\begin{equation}
H_\bullet(\widetilde{\mathcal{G}})\cong \lim_{m \rightarrow \infty}
   \HCE_\bullet(\g(W \oplus \CC^{2m})), \quad H_\bullet(\widetilde{\mathcal{G}})\cong \lim_{m \rightarrow \infty}
   \HCE_\bullet(\g(W \oplus \Pi \CC^m)).
\end{equation}
\end{theorem}
\begin{remark}
It might be possible to further generalize this result
to (certain) cases where both the even and the odd part have dimension going to
infinity, but that creates additional technical difficulties that we
prefer to avoid (as we do not need such generality). In any case, they
are the same difficulties that arise in the original uncurved setting
of \cite{Kncsg} (note that in \emph{op.~cit.} only the even case is
considered).
\end{remark}
\begin{proof}
  The corresponding result for the uncurved graph complex and vector
  fields vanishing at the origin was established by Kontsevich
  \cite{Kncsg} and his proof carries over to the present context, up
  to some technical difficulties created by the fact that we are in
  the super context, where $W_0$ and $W_1$ can both be nonzero. Define
  a nonnegative grading on $\CE^\bullet(\g_r(W))$ called
  \emph{weight}, which is the sum of the homological grading and the
  degree of polynomial coefficients of the vector fields (this will
  correspond, on graphs, to the number of half-edges).

  The main tool is that the inclusion of subcomplexes
\begin{equation} \label{invincleq}
\CE^\bullet(\g_r(W \oplus \CC^{2m}))^{\mathfrak{osp}(W \oplus
    \CC^{2m})} \into \CE^\bullet(\g_r(W \oplus \CC^{2m}))
\end{equation}
is asymptotically a quasi-isomorphism, and similarly replacing
$\CC^{2m}$ with $\Pi \CC^m$.  By this, we mean that in weights $\leq
N$, there exists $M$ such that, if $m \geq M$, then the inclusion is
an isomorphism on homology in weights $\leq N$.

We carry out the argument with $\CC^{2m}$; replacing this by $\Pi
\CC^m$ will not affect anything. Let $U_m := W \oplus \CC^{2m}$.  First note that $\mathfrak{osp}(U_m)
\subseteq \g_r(U_m)$ is the subspace of
linear vector fields and hence acts trivially on the cohomology of
$\CE^\bullet(\g_r(U_m))$.  Hence, the statement would
follow if it were true that $\mathfrak{osp}(U_m)$ acted
completely reducibly on $\CE^\bullet(\g_r(U_m))$.  This
is not, in general, true (when $W$ is not purely even); however, it
follows from Lemma
\ref{superredlem} below that,
for each $N \geq 0$, there exists $M$ so that $m \geq M$
implies that the weight $\leq N$ part of $\CE^\bullet(\g_r(U_m))$ is completely reducible as an $\mathfrak{osp}(U_m)$-representation. This is sufficient
to deduce that \eqref{invincleq} is a quasi-isomorphism in weights $\leq N$.

To complete the proof, it remains to relate the invariant subcomplex
to graphs.  This part of the argument is nearly identical to
\emph{op.~cit.}, so we will be brief.  Associated to each graph is an
$\mathfrak{osp}(W)$-invariant element of $\CE_\bullet(\g_r(W))$, as
described in the previous subsection. However, the resulting map
$\mathcal{G}_r \to \CE_\bullet(\g_r(W))$ does not linearly extend to a
map of complexes. Instead, if we attach to the dual of a graph in
$\mathcal{G}_r^*$ a canonical element of
$\CE^\bullet(\g_r(W))^{\mathfrak{osp}(W)}$, as explained in
\emph{op.~cit.}, one obtains a canonical map of complexes
$(\mathcal{G}_r^*, d^*) \to \CE^\bullet(\g_r(W))^{\mathfrak{osp}(W)}$.
In more detail, under the identification $\mathcal{G}^*_r \cong
\mathcal{G}_r$ using the basis of ribbon graphs (with fixed
orientations), this map sends each oriented ribbon graph $\Gamma$ to
the corresponding functional which contracts elements of
$\CE_{|V_\Gamma|}(\g_r(W))$ using the symplectic form $W \otimes
W\rightarrow \CC$, similarly to the construction of \S \ref{ccs}.  That is,
we view $\CE_{|V_\Gamma|}(\g_r(W)) \subset \Hom(W, \hat{T}(W))$ as a subspace
of $\hat{T}(W)$ using the symplectic form, and contract with a
permutation of $\omega^{\otimes |E_\Gamma|}$, where $\omega \in W
\otimes W$ is the inverse to the symplectic form, so that the copy of
$\omega$ corresponding to each edge contracts the corresponding pair
of half-edges.  For details on how to prove this indeed yields a map
of complexes, see, e.g., \cite[Theorem 4.10]{HaLacc}, and also
\cite{HamsaKt}.

By the fundamental theorems of invariant theory \cite{Wftit} (see
also, e.g., \cite{Howeftit}) and the super generalization found in
\cite{Seracit}, following the reasoning in the proof of Lemma
\ref{superredlem} below, it follows that this map of complexes is
asymptotically an isomorphism: for every fixed $N \geq 0$, there
exists $M \geq 0$ so that, when $\dim W \geq M$, this map is an
isomorphism if we restrict to graphs with at most $N$ edges, and hence
elements of $\CE^\bullet(\g_r(W))^{\mathfrak{osp}(W)}$ of weight $\leq N$.

This construction goes through completely analogously for commutative
graphs and the Lie algebra $\g$.
\end{proof}
As in the proof above, let $U_m := W \oplus \CC^{2m}$.  Also define $U_m' := W \oplus \Pi \CC^m$.
\begin{lemma}\label{superredlem}
For all $N \geq 0$, there exists $M \geq 0$ such that, when $m \geq M$, $U_m^{\otimes N}$ and $(U_m')^{\otimes N}$ are completely reducible $\mathfrak{osp}(U_m)$- and $\mathfrak{osp}(U_m')$-modules, respectively.
\end{lemma}
\begin{proof}
  We carry out the argument for $U_m$; the same applies for $U_m'$.
  Suppose that one has subrepresentations $0 \neq \tau_1 \subsetneq
  \tau_2 \subseteq U_m^{\otimes N}$.  We wish to show that the
  inclusion $\tau_1 \into \tau_2$ splits.  By adjunction, we
  equivalently need to show that the composition $\CC \into \tau_1^*
  \otimes \tau_1 \into \tau_1^* \otimes \tau_2$ splits.  Since
  $U_m^{\otimes N}$ is self-dual using the pairing, $\tau_1^* \otimes
  \tau_2 \subseteq U_m^{\otimes 2N}$. This means that it suffices to
  show that the inclusion of the invariant part, $(U_m^{\otimes
    2N})^{\mathfrak{osp}(U_m)} \into U_m^{\otimes 2N}$, splits. Let us
  replace $2N$ by $N$ for convenience.

  This last fact follows using the fundamental theorems of invariant
  theory.  Namely, by \cite{Seracit}, all of the $\mathfrak{osp}(U_m)$-invariants in the tensor algebra $T(U_m)$ are tensor products of the
  pairing on $U_m$ with an invariant related to the determinant, whose
  tensor weight (i.e., number of tensor components) is a function of
  $m$ that goes to infinity when $m$ does.  Hence, for large enough
  $m$, all of the invariants in weight $U_m^{\otimes \leq N}$ are
  spanned by tensor products of the pairing on $W$. By the classical
  second fundamental theorem of invariant theory \cite{Wftit} (see
  also \cite{Howeftit}), for large enough $m$, there are no nontrivial
  relations between these invariants.  The same argument
  applied to $U_m^*$ shows that the coinvariants have the same
  description. Moreover, for large enough $m$, the invariants of
  $U_m^{\otimes \leq N}$ and invariants of $(U_m^*)^{\otimes \leq N}$
  have a perfect pairing.  This implies that the composition
  $(U_m^{\otimes \leq N})^{\mathfrak{osp}(U_m)} \into U_m^{\otimes N}
  \onto (U_m^{\otimes N})_{\mathfrak{osp}(U_m)}$ is an isomorphism for
  large enough $m$.
\end{proof}

Unlike the case for the usual graph homology, it is possible to give a
complete calculation of the homology of the curved graph complex and
curved ribbon graph complex:
\begin{theorem}
\label{acythm}\
\begin{enumerate}
\item[(a)]
The homology of the complex  $(\widetilde{\mathcal{G}_r}, d)$ is identified with the space of
  formal linear combinations of graphs all of whose connected
  components are graphs whose vertices have valence one with the exception of at most a
  single vertex, which has odd valence.
\item[(b)]
The homology of the complex $(\widetilde{\mathcal{G}}, d)$ is identified with the space of
  formal linear combinations of graphs each of whose connected
  components is the connected graph with one edge and two vertices.
\end{enumerate}
\end{theorem}
We can reinterpret this theorem as follows. First, it is enough to
compute the homology of the subcomplex of connected nonempty graphs,
$\widetilde{\mathcal{G}_{r,c}} \subset \widetilde{\mathcal{G}_r}$ and
$\widetilde{\mathcal{G}_c} \subset \widetilde{\mathcal{G}}$.  We call
these complexes the \emph{connected} ribbon graph complex and the
connected graph complex.  Note that $\widetilde{\mathcal{G}_r} \cong
\Sym \widetilde{\mathcal{G}_{r,c}}$ and $\widetilde{\mathcal{G}} \cong
\Sym \widetilde{\mathcal{G}_c}$.  Then, the desired result is that the
homology of the former is identified with linear combinations of
star-shaped graphs with an odd number of edges and at most a single
vertex of valence $\geq 2$, and the latter is one-dimensional and
spanned by the connected graph with two vertices and a single edge.

In other words, the theorem states that
$\widetilde{\mathcal{G}_{r,c}}$ and $\widetilde{\mathcal{G}_c}$ are
quasi-isomorphic to the subcomplexes spanned by connected graphs with
at most a single vertex of valence $\geq 2$.  These complexes are
identical with the deformation complexes of Theorem \ref{cyccurvtrivthm}
for the one-dimensional curved $A_\infty$ and $L_\infty$ algebras with
curvature $c$ satisfying $(c,c) = 1$, and all higher operations zero.
That is, they are the graded vector spaces of noncommutative
symplectic vector fields and ordinary symplectic vector fields on the
odd one-dimensional symplectic vector space $\CC \cdot c$, equipped
with the differential $\ad(\frac{\partial}{\partial c})$ (which turns out to be zero).
\begin{remark}
  The stable homology of the Lie algebra of symplectic vector fields
  has been computed by Guillemin and Shnider in \cite{GSssrc}, and thus,
  part (b) of the above theorem could be deduced from their
  calculation, taking into account Theorem \ref{GFgraph}. However we
  include this result for completeness, and because the argument we
  use in part (a) essentially extends without change to this case.
\end{remark}

\begin{proof}[Proof of Theorem \ref{acythm}]
    Call a graph a \emph{line segment} if it is topologically a line
  segment, i.e., it is connected, and either it is a single vertex
with no edges, or all of its vertices have valence
  two except for two vertices, which have valence one.  Given a
  connected graph $\Gamma$, let us call a vertex $v \in V_\Gamma$
  \emph{exterior} if it either has valence at most one, or one of the
  connected components of $\Gamma \setminus v$ is a line segment. Call
  all other vertices \emph{interior}.

  We will make use of a filtration on the connected (ribbon) graph
  complex given by the number of interior vertices in the graph.  We
  call this the \emph{interior vertex filtration}.  The associated
  graded complex is identified, as a vector space, with the (ribbon)
  graph complex, and with the differential which is almost the same,
  but only contracts edges incident to at least one exterior
  vertex. Moreover, the associated graded complexes $\gr
  \widetilde{\mathcal{G}_{r,c}}$ and $\gr \widetilde{\mathcal{G}_c}$
  are graded not only by number of interior vertices, but by the
  subgraph $\Gamma_0 \subseteq \Gamma$ obtained by restricting to
  interior vertices and edges which are incident only to interior
  vertices (allowing here also the empty graph and the graph with a
  single vertex and no edges).  It is clear that $\Gamma_0$ is
  connected (possibly empty).  Let $\gr_{\Gamma_0}
  \widetilde{\mathcal{G}_{r,c}}$ or $\gr_{\Gamma_0}
  \widetilde{\mathcal{G}_c}$ denote the resulting subcomplex graded by
  $\Gamma_0$.

  (a) We claim that, when $\Gamma_0$ contains at least one edge, then
  $\gr_{\Gamma_0} \widetilde{\mathcal{G}_{r,c}}$ is acyclic.

  Assume that $\Gamma_0$ contains an edge.  Pick a half-edge $h$ of
  $\Gamma_0$, and let $v$ be the incident vertex.  We construct from
  $v$ and $h$ a contracting homotopy $s$ on $\gr_{\Gamma_0}
  \widetilde{\mathcal{G}_{r,c}}$.  Namely, for every oriented graph
  $\Gamma$ whose subgraph on internal vertices is $\Gamma_0$, let
  $s \Gamma$ be the graph obtained from $\Gamma$ by adding a new
  univalent vertex together with an edge connecting it to $v$. The
  resulting new half-edge incident to $v$ is, in the cyclic ordering
  at $v$, one half-edge counterclockwise away from $h$.  Pick the
  orientation on the new graph $s \Gamma$ so that $\Gamma$ is one of
  the summands of $d(s \Gamma)$.

  We claim that $sd + ds = \Id$.  It suffices to show that, for every
  oriented graph $\Gamma$ as above, $sd(\Gamma) + ds(\Gamma)=\Gamma$.
  In turn, it suffices to show that, for every edge $e \in E_\Gamma$,
  $s d_e \Gamma = - d_e (s \Gamma)$.  This follows immediately from
  our definition of $s$.

Hence, we deduce that $\gr_{\Gamma_0} \widetilde{\mathcal{G}_{r,c}}$
is acyclic, as claimed.

Next, we compute $\gr_{\Gamma_0} \widetilde{\mathcal{G}_{r,c}}$ where
$\Gamma_0$ is either empty or is the graph with a single vertex and no edges.
Call the first graph ``$\emptyset$'' and the second one ``$\pt$.''

First, $\gr_{\pt} \widetilde{\mathcal{G}_{r,c}}$ consists of star-shaped
graphs with a single vertex of valence $\geq 3$, with the usual graph
differential except that we do not allow to contract an edge that
would result in a graph without a vertex of valence $\geq 3$, i.e., an
edge which is incident to a vertex of valence $1$ and a vertex of
valence $3$. Consider the filtration by the valence of
the interior vertex. The associated graded complexes have homology
which is one-dimensional, concentrated in the part where all vertices
but one have valence $1$.  Moreover, such graphs are actually zero
when the node has even valence (because a cyclic symmetry can reverse
the orientation).  Hence, the associated spectral sequence computing
$H_*(\gr_{\pt} \widetilde{\mathcal{G}_{r,c}})$ collapses at the first
page, and the resulting homology is spanned by the star-shaped graphs
with a single vertex of odd valence $\geq 3$ and with all other
vertices of valence $1$.

Next, $\gr_\emptyset \widetilde{\mathcal{G}_{r,c}}$ is the subcomplex of line
segments, whose homology is one-dimensional and spanned by the line segment
with two vertices.

We deduce that the first page of the spectral sequence of the interior
vertex filtration on $\widetilde{\mathcal{G}_{r,c}}$ is concentrated
in the part with $\leq 1$ interior vertices.  The part in degree $1$
is the span of star-shaped graphs with one vertex of odd valence $\geq
3$ and the other vertices of valence $1$, and the part in degree $0$
is the span of the line segment with two vertices.  Since all of these
graphs have only odd numbers of edges, it is clear that the spectral
sequence collapses at the first page, and the graphs above span the
homology of $\widetilde{\mathcal{G}_{r,c}}$, as desired.

(b) The same argument as above applies in this case, except that,
since our graphs no longer have cyclic orderings of half edges at
vertices, we modify the construction of the contracting homotopy $s$
(used to show that $\gr_{\Gamma_0} \widetilde{\mathcal{G}_{r,c}}$ is
acyclic when $\Gamma_0$ contains an edge) accordingly. Namely, we
remove the condition that the new edge be next to $h$ in the
counterclockwise cyclic ordering of half edges at the vertex $v$.
Everything else goes through without change, except that now the
star-shaped graphs with a single vertex of valence $\geq 1$ are all
zero except for the one with only two vertices.  This implies the
desired result. \qedhere

\end{proof}
We see that the complex of simply-connected graphs (which splits off
as a direct summand) carries all of the homology of the complexes
$(\widetilde{\mathcal{G}_r}, d)$ and $(\widetilde{\mathcal{G}}, d)$,
and we obtain the following result.
\begin{corollary}\label{curveduncurved}
  The inclusions of complexes $(\mathcal{G}_r,
  d)\subset(\widetilde{\mathcal{G}_r}, d)$ and $(\mathcal{G},
  d)\subset(\widetilde{\mathcal{G}}, d)$ induce the zero maps on
  homology.
\end{corollary}
\begin{remark}
  It is natural to ask whether the nontrivial homology classes in the
  curved graph complexes $\widetilde{\mathcal{G}_r}$ and
  $\widetilde{\mathcal{G}}$ are detected by curved $A_\infty$- and
  $L_\infty$-algebras. The answer is yes; in fact, they are detected by
  one-dimensional algebras.  For the $A_\infty$ case, let $\Gamma(i)$
  be the star-shaped graph whose central vertex has valence $2i+1$ and
  $V(i)$ be the one-dimensional cyclic curved $A_\infty$-algebra
  spanned by an even vector $c$ such that $( c,c)=1$, $m_0=c$,
  $m_{2i}(c,\ldots,c)=c$, and all other $A_\infty$-products are
  zero. Then it is easy to see that the cycle
  $[V(i)]\in\widetilde{\mathcal{G}_r}$ is homologous to
  $\pm\frac{1}{2i+1}\Gamma(i)$ (the sign depends on the choice of an
  orientation on $\Gamma(i)$).  Next, let $\Gamma(0)$ be the connected
  graph with one edge and two vertices.  Then, we may let $V(0)$ be
  the one-dimensional cyclic curved $A_\infty$-algebra spanned by an
  even vector $c$ such that $m_0=c$, $( c, c ) = 1$, and all the
  higher $A_\infty$-products are zero.  Then,
  $[V(0)]\in\widetilde{\mathcal{G}_r}$ is again homologous to
  $\pm\frac{1}{2}\Gamma(0)$.  The same statement holds for the
  $L_\infty$-setting, if we now let $\Gamma(0)$ be an ordinary (not
  ribbon) connected graph also with two vertices and one edge, and
  consider the one-dimensional curved cyclic $L_\infty$-algebra,
  $V(0)$, with curvature $c$ satisfying $( c, c ) = 1$ and all higher
  operations equal to zero: $[V(0)]\in\widetilde{\mathcal{G}_r}$ is
  homologous to $\pm\frac{1}{2}\Gamma(0)$.
\end{remark}
\section{Homology of Lie algebras of vector fields and stability maps}\label{unstabsec}
In order to refine Theorem \ref{acythm}, we recall first a more
general way to view the construction of characteristic classes.
Let $\mathfrak{g}$ be a formal dgla.
Consider a Maurer-Cartan element of $\mathfrak{g}$, i.e., an element
$x \in \Pi\mathfrak{g}$ satisfying $dx + \frac{1}{2} [x,x] = 0$.
Then, $e^x = \sum_{i \geq 0}x^{i}/i!$ defines a cycle in
 $\CE_\bullet(\mathfrak{g})$, and hence a homology class of even
total degree.  In the situation where the differential $d$ on
$\mathfrak{g}$ is zero, using the homological degree $|S^{i}
\Pi\mathfrak{g}| = i$, each element $x^{ i}$ itself is a cycle, and we
obtain \emph{unstable} characteristic classes $[x^{ i}] \in
\CE_i(\mathfrak{g})$.

Moreover, if $\mathfrak{h} \subseteq \mathfrak g_0$ is a pronilpotent
Lie subalgebra of the even part of $\mathfrak{g}$, then there is
defined a notion of \emph{gauge equivalence} of Maurer-Cartan elements
corresponding to the adjoint action of the Lie group of the Lie
algebra $\mathfrak{h}$; then it follows that if two Maurer-Cartan
elements are gauge equivalent by a gauge in $\mathfrak{h}$, then their
characteristic classes are homologous. (One can more generally take
$\mathfrak{h} \subseteq \mathfrak{g}_0$ to be a Lie subalgebra which
is the Lie algebra of a pro-Lie group).  This statement, as well as a
more detailed treatment of characteristic classes, can be found in,
e.g., \cite{Hamccms}.

Returning to the situation of a (cyclic or curved) $L_\infty$- or
$A_\infty$-algebra $V$, the corresponding element of the Lie algebra
$\Der^0(\hat{S} (\Pi V^*)), \Der(\hat{T}(\Pi V^*))$, etc., defines a
canonical homology class.

The relationship to the aforementioned characteristic classes is
Kontsevich's result that the limit as the dimension of $V$ goes to
infinity of the Lie homology of $\CDer^0(\hat{S} (\Pi V^*))$ or
$\CDer^0(\hat{T}(\Pi V^*))$ is the completion (by number of edges) of
the homology of the graph complexes $(\mathcal{G}, d)$. In the curved
situation the relevant result is Theorem \ref{GFgraph}.

Furthermore, given $V$, and any inner product space $W$,
 we can form the trivial extension $V \oplus
W$ where all multiplications with the second factor are zero. This
induces maps
\begin{gather} \label{vplieeqn}
\varphi_{V,W}: \CE_\bullet(\CDer(\hat{S}(\Pi V^*)))
\rightarrow \CE_\bullet(\CDer(\hat{S}(\Pi (V \oplus W)^*))), \\
\label{vpasseqn}
\varphi_{V,W}: \CE_\bullet(\CDer(\hat{T}(\Pi V^*)))
\rightarrow \CE_\bullet(\CDer(\hat{T}(\Pi (V \oplus W)^*))),
\end{gather}
and similarly the restrictions $\varphi_{V,W}^0 :=
\varphi_{V,W}|_{\CE_\bullet(\CDer^0(\hat{S}(\Pi V^*)))}$ or
$\varphi_{V,W}^0 := \varphi_{V,W}|_{\CE_\bullet(\CDer^0(\hat{T}(\Pi
  V^*)))}$.  It is well known (and easy to check) that Kontsevich's
construction (\S \ref{ccs}) is obtained from the above construction in
the limit: the image of the unstable characteristic cycle $\sum_i
\xi^{i}/i! \in \CE_\bullet(\CDer(\hat{S}(\Pi V^*)))$ under
$\varphi_{V,W}^0$ as $\dim W_0 \rightarrow \infty$ (for fixed $W_1$)
identifies with the characteristic cycle on $(\mathcal{G}, d)$ given
in $\S \ref{ccs}$; similarly if we fix $W_0$ and let $\dim W_1 \to
\infty$.  In the curved setting, by Theorem \ref{acythm}, for each
fixed degree $i$, if we fix $W_1$, then for large enough $\dim W_0$,
$\varphi_{V,W}$ induces a projection on homology,
$$\HCE_i(\CDer(\hat{S}(\Pi V^*)))
\onto \begin{cases} \CC, & \text{if $i$ is even}, \\ 0, & \text{if $i$
    is odd}.\end{cases}$$
The same is true if we fix $W_0$ and consider
large enough $\dim W_1$.  A similar result is true in the associative
version, where now we project onto the span of graphs whose connected
components are stars with odd valence as in Theorem \ref{acythm}.

 Using Theorem \ref{cyccurvtrivthm}, we can prove an unstable
 analogue of Corollary \ref{curveduncurved}, which gives information
 about the maps $\varphi_{V,W}$ for all $W$ with $W_0 \neq 0$:
\begin{theorem}\label{unscurveduncurved}
  If $V$ and $W$ are inner product spaces and $W_0 \neq
  0$, then the compositions
\begin{gather} \label{lstabmaps}
\CE_i(\CDer^0(\hat S(\Pi V^*))) \to \CE_i(\CDer(\hat S(\Pi V^*))) \mathop{\to}^{\varphi_{V,W}} \CE_i(\CDer(\hat S(\Pi (V \oplus W)^*))), \\ \label{astabmaps}
\CE_i(\CDer^0(\hat T(\Pi V^*))) \to \CE_i(\CDer(\hat T(\Pi V^*))) \mathop{\to}^{\varphi_{V,W}} \CE_i(\CDer(\hat T(\Pi (V \oplus W)^*)))
\end{gather}
are zero on homology.
\end{theorem}
Before we prove the theorem, we first prove a lemma which may be
interesting in itself.  Given an element $c \in V$ and a $L_\infty$- or
$A_\infty$-structure given by a Maurer-Cartan element $\xi \in
\CDer^0(\hat S(\Pi V^*))$ or $\xi \in \CDer^0(\hat T(\Pi V^*))$, we
say that $c$ is \emph{central} if $[c, \xi] = 0$. In particular, for
an ordinary Lie algebra, $c$ is an element satisfying $\{c, v\}=0$ for
all $v \in V$, and for an ordinary associative algebra, $c$ satisfies
$c \cdot v = v \cdot c$ for all $v \in V$.
\begin{lemma}
Suppose that $V$ is an (uncurved) $L_\infty$- (or $A_\infty$-) algebra with a nonzero even central element $c \in V_0$.
Then, the image of the resulting (unstable) characteristic class of $V$ under the appropriate map,
\begin{gather}
\CE_\bullet(\CDer^0(\hat S(\Pi V^*))) \to \CE_\bullet(\CDer(\hat S(\Pi V^*))) \text{ or }\\
\CE_\bullet(\CDer^0(\hat T(\Pi V^*))) \to \CE_\bullet(\CDer(\hat T(\Pi V^*))),
\end{gather}
is zero on homology.
\end{lemma}
\begin{proof}
We first consider the $L_\infty$ case.  Let $\mathfrak{g} :=
  \CDer(\hat{S}(\Pi V^*))$.  Let $\xi
  \in \CDer^0(\hat{S}(\Pi V^*)) \subseteq
  \mathfrak{g}$ correspond to the $L_\infty$-structure on $V$,
  i.e., $\xi$ satisfies $[\xi, \xi] = 0$, and viewed as an element
of $\mathfrak{g}$, $\xi(v)$ has no
  constant term for all $v \in \Pi V^*$.
  Let $\ell$ denote the structure maps for the algebra
  $V$, with $\ell_i: V^{\otimes i} \rightarrow V$ the $i$-th component.

  Now, consider the $L_\infty$-structure $\{\ell^\lambda_i\}$ on $V$
  which is obtained by $\ell^\lambda_i := \ell_i$ if $i \geq 1$, and
  $\ell_0^\lambda = \lambda c$ for $\lambda \in \CC$. Since $c$ is
  central and even, these indeed define $L_\infty$-structures.

  Then, by Theorem \ref{cyccurvtrivthm}, $V$ equipped with
  $\ell^\lambda$ is gauge equivalent to an algebra with all higher
  multiplications equal to zero, and curvature equal to $\lambda
  c$. This gauge equivalence is by vector fields with zero constant
  and linear term, which form a pronilpotent dgla.  Hence, in the
  limit as $\lambda \rightarrow 0$, we see that the characteristic
  class of $(V, \ell^\lambda)$ becomes a boundary, i.e., the original
  characteristic class of $(V, \ell)$ is a boundary in
  $\CE_\bullet(\CDer(\hat S(\Pi V^*)))$, as desired.

  In the $A_\infty$ case, the same argument applies: as before, one
  deforms the uncurved $A_\infty$-structure $\{m_i\}$ to the curved
  structure $\{m^\lambda_i\}$ with $m^\lambda_i = m_i$ for $i \geq 1$
  and $m^\lambda_0 = \lambda c$.  The difference is, by Theorem
  \ref{cyccurvtrivthm}, the resulting structure $(V, m^\lambda)$ is
  gauge equivalent to the algebra described there, which does not have
  all higher operations zero.  Call this algebra structure
  $\{(m')^\lambda_i\}$.  Even though these are nonzero for infinitely
  many $i$, it is still true that, as $\lambda \rightarrow 0$,
  $(m')^\lambda_i \rightarrow 0$, and so we still deduce that the
  original characteristic class was a boundary.
\end{proof}

\begin{proof}[Proof of Theorem \ref{unscurveduncurved}]
  We treat only the $L_\infty$ case, since the $A_\infty$ case is
  identical.  If $V$ is an uncurved $L_\infty$-algebra, then $V \oplus
  W$ is an $L_\infty$-algebra with a nonzero central element, namely,
  any nonzero element of $W_0$.  Hence, the image of the resulting
  characteristic class under \eqref{lstabmaps} is zero.

However, in general, not all homology classes of
$\CE_\bullet(\CDer^0(\hat{S}(\Pi V^*)))$ are obtainable in the above manner.
To fix this problem, we can consider nontrivial coefficients.  Given
any cdga $\mathfrak{a}$, we may consider
$\mathfrak{a}$-linear $L_\infty$-algebra structures on $U := V
\otimes_\CC \mathfrak{a}$, for a fixed finite-dimensional
space $V$.  Denote $U^* := \Hom_{\mathfrak{a}}(U,
\mathfrak{a})$.  Then, these algebra structures on $U$ are, by
definition, Maurer-Cartan elements of
$\Der^0(\hat{S}_{\mathfrak{a}}(\Pi U^*))$, i.e., odd elements $f$
satisfying $df + \frac{1}{2} [f, f] = 0$, where $d$ is the
differential induced by the differential on $\mathfrak{a}$. If we fix
an inner product on $V$, we obtain an induced $\mathfrak{a}$-linear
inner product on $U$ (of the form $U \otimes_{\mathfrak{a}} U
\rightarrow \mathfrak{a}$), and cyclic $L_\infty$-algebras of this
form are given by elements of $\CDer^0(\hat{S}_{\mathfrak{a}}(\Pi
U^*))$.

As above, we obtain an unstable characteristic class of
$\CE_\bullet(\CDer^0(\hat{S}_{\mathfrak{a}}(\Pi U^*)))$.  The above argument
shows that the images of such classes under the tensor product of \eqref{lstabmaps} with $\Id_{\mathfrak{a}}$ are boundaries.

To conclude the proof, we will use a universal example.
Quite generally, if $\mathfrak{h}$ is a formal dgla and
$\mathfrak{a}$ a cdga, then $\CE^\bullet(\mathfrak{h}) =
S\Pi \mathfrak{h}^*$ is naturally a cdga, and
Maurer-Cartan elements of $\mathfrak{h} \otimes \mathfrak{a}$
identify with cdga morphisms $\CE^\bullet(\mathfrak{h})
\rightarrow \mathfrak{a}$.
If we set $\mathfrak{a}
:= \CE^\bullet(\mathfrak{h})$, then the identity map
$\Id_{\CE^\bullet(\mathfrak{h})}$ yields a Maurer-Cartan element of
$\mathfrak{h} \otimes \mathfrak{a}$, and the resulting cycle $\sum_i
\xi^{ i} / i! \in \widehat{\CE_\bullet(\mathfrak{h}) \otimes \mathfrak{a}}
\cong \End(\CE^\bullet(\mathfrak{h}))$ is nothing but the identity map.

Applying this construction to the case $\mathfrak{h} =
\CDer^0(\hat{S}(\Pi V^*))$ and $\mathfrak{a} =
\CE^\bullet(\CDer^0(\hat{S}(\Pi V^*)))$, we deduce from the above that
the $\mathfrak{a}$-linear maps $\widehat{\varphi_{V,W} \otimes
  \Id_{\mathfrak{a}}}$ send the cycle $\xi$ corresponding to the
identity element of $\End(\CE^\bullet(\mathfrak{h}))$ to a
boundary. This implies that \eqref{lstabmaps} itself is zero on
homology, as desired.
\end{proof}

\section{Operadic generalization}
In this section we sketch an operadic generalization of our main
results, from the associative and Lie cases to more general
settings. As we will show in the following section, these include
Poisson, Gerstenhaber, BV, permutation, and pre-Lie algebras: see
Examples \ref{pgexam}, \ref{bvexam}, and \ref{permexam} in the
following section.

We may think of curved (cyclic) $A_\infty$- and $L_\infty$-algebras as
arising from the following construction, which we think of
heuristically as a type of ``Koszul duality'' between operads
governing curved algebras and those governing unital algebras (we do
not attempt to make this description precise).
\begin{remark}
  For a somewhat related result, see \cite{HiMickdt}, where resolutions
  for operads of unital algebras are constructed by defining a Koszul
  dual curved cooperad and performing a version of the cobar
  construction. Here, we will not make use of the notion of curved
  (co)operads defined in \cite{HiMickdt}, and will only use ordinary (dg)
  operads.
\end{remark}
Let $\mathcal{O}$ be a (cyclic) dg operad \cite{GKcoch}, which we
assume to be unital with unit $I \in \mathcal{O}(1)$. Moreover, we
will assume throughout that each $\mathcal{O}(i)$ is a
finite-dimensional $\ZZ/2$-graded vector space (this isn't really
essential, but it makes dualization less technical, and includes all
operads we have in mind. Properly speaking, we view $\mathcal{O}(i)$
as a formal space, and everything generalizes to the
infinite-dimensional formal setting.)  Let $m_0 \in \mathcal{O}(0)$ be
an element (corresponding to a ``0-ary'' operation).  Recall that
every (cyclic) operad is an $\SS$ ($\SS_+$)-module, where an
$\SS$-module is defined as a collection $\{V_m\}_{m \geq 0}$ of
$S_m$-modules for all $m \geq 0$, and an $\SS_+$-module is an
$\SS$-module where each $V_m$ is actually a module over $S_{m+1}$ (the
underlying $\SS$-module is obtained using the inclusion $S_m \subseteq
S_{m+1}$ of permutations fixing $m+1$).

We will now consider $\mathcal{O}$ as a nonunital operad and
perform the cobar construction \cite{GiKa}. Namely, let $\mathcal{O}^* := \{ \mathcal{O}(m)^* \}$ be the dual
$\SS$ (or $\SS_+$)-module. Let $C(\mathcal{O})$ be the free operad
generated by $\Pi \mathcal{O}^*$, equipped with a differential
$d_{C(\mathcal{O})}$ obtained as follows.  For every $j \geq 0$ and
all $1 \leq i \leq k$, there is an operadic composition map
\begin{equation}
\circ^{k,j}_i: \mathcal{O}(k) \otimes \mathcal{O}(j) \rightarrow \mathcal{O}(j+k-1),
\end{equation}
which corresponds to plugging the element of $\mathcal{O}(j)$ into the
$i$-th input of the element of $\mathcal{O}(k)$.  Let
$(\circ^{k,j}_i)^*$ be the linear dual to the above map.
Then, we
define the differential $d_{C(\mathcal{O})}$ to be the
unique extension to a derivation of the operation
\begin{equation}
d_{C(\mathcal{O})}|_{\mathcal{O}(i)^*} = d_{\mathcal{O}}^* +
\bigoplus_{j,k} (\circ^{k,j}_i)^*.
\end{equation}
Now, in the case that $\mathcal{O}$ is the ordinary (non-dg) operad
governing \emph{unital} associative or commutative algebras, then
$C(\mathcal{O})$ is a dg operad with the property that
graded (dg) algebras $V$ over $C(\mathcal{O})$, equipped
with zero differential, are the same as curved $A_\infty$- or
$L_\infty$-algebras. In this case, $m_0 \in \mathcal{O}(0)$ is the unit of the
multiplication, and the curvature of a graded algebra $V$ over $C(\mathcal{O})$ is the image of $m_0^* \in \mathcal{O}(0)^*$ in $V$.

Next, suppose that
$\mathcal{O}'$ is the suboperad of $\mathcal{O}$ in positive arity,
i.e., $\mathcal{O}'(0) = 0$ and $\mathcal{O}'(i) = \mathcal{O}(i)$ for $i \geq 1$.
 Further
suppose that $\mathcal{O}'$ is a Koszul operad with augmentation
$\overline{\mathcal{O}'}$ (i.e., $\overline{\mathcal{O}'}$ is
a suboperad such that $\mathcal{O}'(1) = \overline{\mathcal{O}'}(1) \oplus \CC \cdot \Id$, and
$\overline{\mathcal{O}'}(i) = \mathcal{O}'(i)$ for $i \neq 1$).
Then, $C(\overline{\mathcal{O}'})$ yields a resolution of the
Koszul dual operad $(\mathcal{O}')^!$ (this last fact is equivalent to
Koszulity; see \cite{GiKa}).  \emph{Caution:} This operad $C(\overline{\mathcal{O}'})$ is sometimes called the cobar construction on the unital operad
$\mathcal{O}'$ itself, but
here we consider cobar constructions only on nonunital operads.

We think of an algebra over $C(\mathcal{O})$ (with zero differential)
as a certain type of \emph{curved} version of an
$((\mathcal{O}')^!)_\infty$-algebra, as opposed to an algebra over
$C(\overline{\mathcal{O}'})$, which, in the case $\mathcal{O}'$ is
Koszul, is the same as an ordinary
$((\mathcal{O}')^!)_\infty$-algebra.  (More generally, given a dg
space $V$ with possibly nontrivial differential, one could view a
$C(\mathcal{O})$-algebra structure on it as a dg generalization of a
curved $((\mathcal{O}')^!)_\infty$-algebra; thus one could speak, for
instance, about a curved dg $A_\infty$- or $L_\infty$-algebra. This
would be a Maurer-Cartan element in the appropriate differential
graded Lie algebras of (formal, noncommutative) vector fields rather
than simply square-zero odd derivations.)

Let us return to the general setting of an arbitrary (cyclic) operad
$\mathcal{O}$. The analogues of Theorems \ref{curvtrivthm} and
\ref{cyccurvtrivthm} are then the following.
\begin{definition}
  A \emph{weakly unital multiplication} $m_2 \in \mathcal{O}(2)$
  is an element such that the image of the maps
\begin{equation}\label{wtm2map}
m_2(- \otimes \Id), m_2(\Id \otimes -): \mathcal{O}(0) \rightarrow \mathcal{O}(1)
\end{equation}
lie in the one-dimensional space $\CC \cdot \Id \subseteq \mathcal{O}(1)$, and such that the maps are not both zero.
\end{definition}
If $m_2$ is weakly unital, then
\begin{equation} \label{m2ideq}
m_2(- \otimes \Id) = m_2(\Id \otimes -).
\end{equation}
Indeed, if $m_0 \in \mathcal{O}(0)$ is any element such that
$m_2(m_0 \otimes \Id) = \lambda \cdot \Id$ for some nonzero $\lambda
\in \CC$, then it follows that $m_2(m_0 \otimes m_0) = \lambda\cdot m_0$,
and hence that $m_2(\Id \otimes m_0) = \lambda \cdot \Id$.

Therefore, when $m_2$ is weakly unital, we define $\widetilde{m_2}
\in \mathcal{O}(0)^*$ to be the resulting map \eqref{m2ideq}.

Next, given a $C(\mathcal{O})$ algebra $V$, the $0$-ary operations form a map $\gamma: \mathcal{O}(0)^* \rightarrow V$, which we call the \emph{$0$-ary structure
map}.
%  and we can rewrite this as a canonical element
% \begin{equation}
% \gamma \in \mathcal{O}(0) \otimes V,
% \end{equation}
% which we call the \emph{$0$-ary element} of $V$.
\begin{definition}
Given an operad $\mathcal{O}$ with weakly unital multiplication $m_2$,
a $C(\mathcal{O})$ algebra $V$ with $0$-ary structure map $\gamma$
is said to be
\emph{nontrivially curved with respect to $m_2$}
if the element $\gamma(\widetilde{m_2})$ is nonzero.
\end{definition}
We call the element $\gamma(\widetilde{m_2})$ the
\emph{curvature} of $V$ (for the element $m_2$).

Then, Theorems \ref{curvtrivthm} and \ref{cyccurvtrivthm}
generalize as follows.  Fix a graded vector space $V$ (with zero
differential).  If $\mathcal{O}$ is any dg operad, we denote a
$\mathcal{O}$-algebra structure on $V$ as a pair $(V, \phi)$ for the
algebra structure, where given any $o \in \mathcal{O}(k)$, $\phi(o) \in
\Hom_\CC(V^{\otimes k}, V)$.
\begin{claim} Fix a graded vector space $V$. \label{opclaim}
\begin{enumerate}
\item[(i)] Let $\mathcal{O}$ be an operad with a weakly unital
  multiplication $m_2$. Let $(V,\phi)$ be a
  $C(\mathcal{O})$-algebra structure with nonzero curvature $c
  \in V$.
  Then, $(V,\phi)$ is gauge equivalent to the algebra $(V,\phi')$ with
  the same $0$-ary structure map $\gamma$,
but with all
  higher multiplications equal to zero.
\item[(ii)] Let $\mathcal{O}$ be a cyclic operad with a weakly unital
  multiplication $m_2$. Let $(V,\phi)$ be a cyclic
  $C(\mathcal{O})$-algebra structure, with nonzero curvature $c
  \in V$.  Let $c' \in V$ be an element such that $(c', c) = 1$.
  Then, $(V,\phi)$ is gauge equivalent to the algebra $(V,\phi')$ with
  the same $0$-ary structure map $\gamma$, but with higher
  multiplications of the form
\begin{equation}\label{opclaim12eq}
\phi'(m)(v_1, \ldots, v_i) = \bigl(\prod_{j=1}^i (c',v_j) \bigr) \cdot (\phi(m)(c, c, \ldots, c), c) \cdot c'.
\end{equation}
\end{enumerate}
\end{claim}
Here, gauge equivalence is defined as follows.  For any
$C(\mathcal{O})$-algebra structure on $V$, one can form the formal dg
$\mathcal{O}$-algebra $\hat{C}_{\mathcal{O}}(V)$, which is defined as the completed
free $\mathcal{O}$-algebra generated by $\Pi V^*$:
\[
\hat{C}_{\mathcal{O}}(V)=\prod_{n=0}^\infty {\mathcal
  O}(n)\otimes_{S_n}(\Pi V^*)^{\otimes n} ,\] equipped with the
differential $d_{\hat{C}_{\mathcal{O}}(V)}$ on
$\hat{C}_{\mathcal{O}}(V)$ defined by the canonical linear map $\Pi V^*
\rightarrow \hat{C}_{\mathcal{O}}(V)$, which is shifted dual to a
restriction of the structure maps $C(\mathcal{O}) \rightarrow
\Hom(V^{\otimes n}, V)$.

Next, assume for the moment that $V$ is equipped with the zero
$C(\mathcal{O})$-algebra structure, i.e.,
$d_{\hat{C}_{\mathcal{O}}(V)} = 0$. We may then form the formal dgla
$\Der(\hat{C}_{\mathcal{O}}(V))$, where the differential is induced by
the differential on $V$ (and is zero in the case where $V$ was merely
a graded vector space, as in the setting of (non-dg) curved
$A_\infty$- or $L_\infty$-algebras considered in earlier sections of
this paper).  It is then a standard fact (see, e.g., \cite[Proposition
2.15]{GJohaii}, Proposition 2.15 where, however, this result is
formulated in the language of coalgebras) that the above yields a
bijection between square-zero odd derivations of
$\hat{C}_{\mathcal{O}}(V)$ and $C(\mathcal{O})$-structures on $V$.
Finally, we define two $C(\mathcal{O})$-structures on $V$ to be gauge
equivalent if the corresponding differentials
$d_{\hat{C}_{\mathcal{O}}(V)}$ and $d_{\hat{C}_{\mathcal{O}}(V)}'$ are
gauge equivalent, i.e., that there exists an even derivation $\xi \in
\Der^0(\hat{C}_{\mathcal{O}}(V))$ with zero constant term such that
$d_{\hat{C}_{\mathcal{O}}(V)}' = e^{\ad \xi}
d_{\hat{C}_{\mathcal{O}}(V)}$.

Similarly, if $\mathcal{O}$ is a cyclic operad, and $V$ is a cyclic
algebra over $\mathcal{O}$ (i.e., an algebra equipped with a
nondegenerate inner product compatible with the cyclic structure on
$\mathcal{O}$), then there is a natural $\ZZ/(m+1)$ action on the
subspace $\Pi V^* \otimes \mathcal{O}(m) \otimes_{S_m} (\Pi V^*)^{\otimes m}
\cong \Hom(\Pi V^*, \hat{C}_{\mathcal{O}}(V)) = \Der(\hat{C}_{\mathcal{O}}(V))$
coming from the $S_{m+1}$-structure on $\mathcal{O}(m)$ and the
compatible inner product on $V$.  Then, we can define the \emph{cyclic
  derivations}, $\CDer(\hat{C}_{\mathcal{O}}(V))$, to be the subspace of
$\ZZ/(m+1)$-invariants in each degree $m$.  Then, a gauge equivalence
$e^{\ad\xi}: d_{\hat{C}_{\mathcal{O}}(V)} \iso d_{\hat{C}_{\mathcal{O}}(V)}'$ is
\emph{cyclic} if $\xi$ is cyclic.

Finally, we explain the main idea of the proof of the claim (without
going into full detail).  We explain only the second part, since it is
more involved. We first form the cyclic deformation complex for the
$C(\mathcal{O})$-structure on $V$
structure with all higher operations $C(\mathcal{O})(i), i > 0$ acting as
zero, and with $0$-ary structure map $\gamma: C(\mathcal{O})(0) = \Pi\mathcal{O}(0)^* \to V$.
This complex is $\CDer(\hat{C}_{\mathcal{O}}(V))$, equipped with the
differential $\ad \xi$, where $\xi \in
\CDer(\hat{C}_{\mathcal{O}}(V))$ is the derivation obtained from $\Pi\gamma^*: \Pi V^* \to \mathcal{O}(0)$, by the formula (for $o \in \mathcal{O}(n)$ and $f_1, \ldots, f_n \in \Pi V^*$):
\begin{equation}
\xi(o \otimes_{S_n} (f_1 \otimes \cdots \otimes f_n)) = \sum_{i=1}^n (o \circ_i \Pi\gamma^*(f_i)) \otimes_{S_{n-1}} (f_1 \otimes \cdots \otimes \hat f_i \otimes \cdots \otimes f_n),
\end{equation}
where $\hat f_i$ denotes omitting $f_i$ from the tensor product.

Then, we again show that this complex is quasi-isomorphic
to the subcomplex spanned by cochains such that $\mathcal{O}(i)$ acts by
multiples of the element
\begin{equation}\label{ceq2}
\epsilon^{i+1}(x_1, \ldots, x_i) := \epsilon(x_1) \cdots \epsilon(x_i) c', \quad \epsilon(v) := (c', v),
\end{equation}
and hence that the original algebra structure is gauge equivalent to
one where all $\geq 1$-ary operations are multiples of the operation
\begin{equation}
\phi'(m)(v_1, \ldots, v_i) = \bigl(\prod_{j=1}^i (c',v_j) \bigr) \cdot c'.
\end{equation}
Finally, as before, we can see that this multiple is $(\phi(m)(c, c, \ldots, c), c)$.

The main technical step in the above is showing that the complex is
quasi-isomorphic to the subcomplex spanned by cochains as in
\eqref{ceq2}. This amounts to a generalization of Lemma \ref{qil}.
The main point is that the map $s_i'$ defined in \eqref{sipfla1} can
be generalized to this setting, as follows.  If $f = o
\otimes_{S_{n+1}} (f_1 \otimes \cdots \otimes f_{n+1}) \in
\mathcal{O}(n) \otimes_{S_{n+1}} (\Pi V^*)^{\otimes (n+1)}$, viewed as
an element of $\CDer(\hat{C}_{\mathcal{O}}(V))$, then $s_n f \in
\mathcal{O}(n+1) \otimes_{S_{n+2}} (\Pi V^*)^{\otimes(n+2)}$ is obtained
as
\[
s_n f = \sum_{j=1}^{n+1} (o \circ_j m_2) \otimes_{S_{n+2}} (f_1 \otimes \cdots \otimes f_{j-1} \otimes \epsilon \otimes f_j \otimes \cdots \otimes f_{n+1}).
\]
Then, one obtains a similar result to \eqref{sipfla3}, which gives the desired conclusion.

Next, Theorem \ref{acythm} generalizes as follows.  Let $\cO$ be a
cyclic operad with a weakly unital multiplication $m_2$, and let
$\mathcal{G}_{\mathcal{O}}$ be the graph complex constructed from the
operad $\mathcal{O}$.  This complex is a generalization of
Kontsevich's graph homology which has many equivalent definitions in
the literature; one definition is via the Feynman transform
construction of \cite{GKmo}.  Namely, consider the naive
$\Det$-twisted modular closure $\underline{\mathcal O}^1$ of $\mathcal
O$ by considering all contraction operations to be zero and all parts
of $\mathcal{O}$ of genus $\geq 1$ to be zero; see, e.g., \cite[\S
2]{ChLadft} for this notion. Then form the Feynman transform operad
$\mathsf{F}\underline{\mathcal O}^1$. The $0$-ary part
$\mathsf{F}\underline{\mathcal O}^1((0))$ is the desired graph
complex.  As before, let $\mathcal{G}_{\mathcal{O},c} \subseteq
\mathcal{G}_{\mathcal{O}}$ be the subcomplex spanned by connected
nonempty graphs; one has $\mathcal{G}_{\mathcal{O}} \cong \Sym
\mathcal{G}_{\mathcal{O},c}$.
\begin{claim}\label{opclaim2}
  The complex $\mathcal{G}_{\mathcal{O},c}$ is quasi-isomorphic to
  the quotient of the subcomplex spanned by graphs with at most one
  vertex of valence $\geq 2$ by the span of line segments with three
  vertices whose central vertex is labeled by $\Id$.\footnote{Note
    that, if $\mathcal{O}(0)$ is even one-dimensional, as in the
    preceding cases of $\mathscr{C}\!\mathit{omm}$ or $\mathscr{A}\!\mathit{ss}$, then these
    line segments with central vertex labeled by $\Id$ are already
    zero. More generally, the span of these line segments is
    isomorphic to $\wedge^2 \cO(0)$, by considering the labels at the
    univalent vertices.}
\end{claim}
In other words, $\mathcal{G}_{\mathcal{O},c}$ is quasi-isomorphic to a
quotient of the deformation complex of a certain canonical
cyclic $C(\mathcal{O})$-algebra (which we think of as a type of curved
$((\mathcal{O}')^!)_\infty$-algebra), as follows.\footnote{This
  interpretation requires that $\mathcal{O}(0)$ be finite-dimensional,
  as we are assuming.  For infinite-dimensional formal
  $\mathcal{O}(0)$, while the claim above still holds, this
  interpretation is technically not available.} Let $V :=
\mathcal{O}(0)^*$. View $V$ as a $\mathcal{O}$-algebra with all $\geq
1$-ary operations trivial, and with $0$-ary structure $\Id:
\mathcal{O}(0)^* \to \mathcal{O}(0)^*$. Since all higher operations
are trivial, any inner product on $V$ is cyclic; we can fix one but it
will not really affect anything.  Then, the deformation complex of $V$
as a $C(\mathcal{O})$-algebra is
$$\CDer(\hat{C}_{\mathcal{O}}(V)) \cong \bigoplus_{m \geq 0} \mathcal{O}(m) \otimes_{S_{m+1}} \mathcal{O}(0)^{\otimes m+1},$$
equipped with the differential $\ad \xi$ where $\xi \in
\CDer(\hat{C}_{\mathcal{O}}(V))$ is the element corresponding to $\Id
\in V \otimes \mathcal{O}(0) \cong \Hom(V, V) \subseteq
\Der(\hat{C}_{\mathcal{O}}(V))$.  By the above,
$\mathcal{G}_{\mathcal{O},c}$ is quasi-isomorphic to the quotient
of this deformation complex by the subcomplex in arity $2$, $\Id
\otimes_{S_2} \mathcal{O}(0)^{\otimes 2} \subseteq \bigoplus_{m \geq 0}
\mathcal{O}(m) \otimes_{S_{m+1}} \mathcal{O}(0)^{\otimes m+1}$.

The proof of Claim \ref{opclaim2} is a direct generalization of the
proof of Theorem \ref{acythm}.  First, we generalize the notion of
interior vertex.  Note that $\mathcal{G}_{\mathcal{O},c}$ is
spanned by oriented graphs of the following form: the vertices are
ordered, and each $m$-valent vertex is labeled by an element of
$\cO(m-1)$.  Next, the half-edges are also ordered.  Many of these
graphs are equivalent: applying a permutation of the half-edges
incident to a given vertex is set equal to applying the corresponding
permutation to the element of $\cO$ labeling that vertex; also,
applying any permutation of the half-edges is set equal to multiplying
by the sign of that permutation. Finally, applying a permutation to
the vertices is the same as multiplying by the sign of that
permutation.

Then, an interior vertex of a graph as above is a vertex which has
valence $\geq 2$ and whose removal does not result in a graph one of
whose connected components consists only of univalent vertices and
bivalent vertices labeled by elements of $\CC \cdot \Id \subseteq
\cO(1)$.  As before, the number of interior vertices defines an
increasing filtration. Moreover, if we choose an arbitrary basis
$\{\Gamma_i\}$ of graphs of $\mathcal{G}_{\mathcal{O},c}$ (and
possibly the empty graph and the graph with a vertex and no edges),
then the associated graded complex is graded by this basis, where
$\gr_{\Gamma_i} \mathcal{G}_{\mathcal{O},c}$ is spanned by graphs
whose restriction to interior vertices yields $\Gamma_i$.

The main step is to generalize the construction of the contracting
homotopy which shows that $\gr_{\Gamma_0}
\mathcal{G}_{\mathcal{O},c}$ is acyclic when $\Gamma_0$ is a graph
containing an edge. To do so, first fix an element $m_0 \in \cO(2)$
such that $m_2 \circ_1 m_0 = m_2 \circ_2 m_0 = \Id \in \cO(1)$.
% We choose the basis $\{\Gamma_i\}$ to consist of graphs.
% Let $\Gamma_0$
% be one such graph which contains an edge.
Fix a half edge $h$ of $\Gamma_0$ based at a vertex $v$.  Let us
choose $v$ to be the last vertex in the ordering of vertices, and $h$
to be the last half-edge in the ordering of half-edges based at $v$.
Suppose $v$ is $m$-valent, and let the label of $v$ be $o_v \in
\cO(m-1)$. Call the half-edges of $v$, in order, $h_1, \ldots, h_m$,
with $h = h_m$.  The contracting homotopy $s$ then acts on any graph
$\Gamma \in \gr_{\Gamma_0} \mathcal{G}_{\mathcal{O},c}$ by first
adding a new half-edge $h_{m+1}$ to $v$ (last in the ordering at $v$).
Then, the label $o_v$ is replaced by $o_v \circ_m m_2$, where $m_2 \in
\cO(2)$ is the weakly unital multiplication. Finally, one adds a new
univalent vertex $y$ to $\Gamma$ (which becomes the last vertex),
labels it by $m_0$, and attaches it to $h_{m+1}$.  It is then
straightforward to verify that $(sd + ds) \Gamma = \Gamma$, and hence
$\gr_{\Gamma_0} \mathcal{G}_{\mathcal{O},c}$, is acyclic.

Finally, we explain the appearance of the quotient by line segments
with three vertices whose central vertex is labeled by $\Id$.  Namely,
these are exactly the graphs with a single vertex of valence $\geq 2$
which, nonetheless, have no interior vertices.  By a generalization of
the arguments in the proof of Theorem \ref{acythm}, $\gr_{\pt}
\mathcal{G}_{\mathcal{O},c}$ is quasi-isomorphic to the quotient of
the subcomplex spanned by star-shaped graphs with a single vertex of
valence $\geq 2$ by the span of these graphs.  On the other hand,
$\gr_{\emptyset} \mathcal{G}_{\mathcal{O},c}$ is quasi-isomorphic
to the subcomplex spanned by graphs with two vertices, each univalent.
Thus, the second page of the spectral sequence for the interior vertex
filtration yields, in degree one, the claimed quotient
complex modulo the graphs with only two vertices, and in degree zero,
the span of graphs with only two vertices and one edge.  The spectral
sequence collapses at the third page to the homology of the whole
quotient subcomplex stated in Claim \ref{opclaim2}, which proves the
result. We omit further details.

Finally, one can deduce an unstable version of Claim \ref{opclaim2},
analogous to Theorem \ref{unscurveduncurved}:
\begin{claim}\label{opclaim3}
The composition
\begin{equation}
\CE_\bullet(\CDer^0(\hat{C}_{\mathcal{O}}(V))) \to \CE_\bullet(\CDer(\hat{C}_{\mathcal{O}}(V)))
\rightarrow \CE_\bullet(\CDer(\hat{C}_{\mathcal{O}}(V \oplus W)))
\end{equation}
is zero on homology for any uncurved cyclic $C(\mathcal{O})$-algebra $V$,
where $W$ is an inner product space with $W_0 \neq 0$.
\end{claim}
We omit the proof, which is obtained by combining the preceding
material with \S \ref{unstabsec}.
\section{Examples and further comments}
In this section we provide some remarks and examples regarding the
material of the previous section.

\subsection{Generalization to the modular case}
Here we briefly sketch how the results of the previous section
can be extended from the
setting of cyclic operads to that of modular operads.

An analogous result to Claim \ref{opclaim2} holds for the
\emph{twisted} version of the $\mathcal O$-graph complex which
corresponds to the Feynman transform ${\mathsf F}\underline{\mathcal
  O}((0))$ (as opposed to ${\mathsf F}\underline{\mathcal O}^1((0))$)
which is the usual version of the graph complex.\footnote{We note that
  the operads ${\mathsf F}\underline{\mathcal O}^d((n))$ are not,
  strictly speaking, modular operads in the sense of \cite{GKmo},
  since the stability condition is not satisfied; however the
  construction of the Feynman transform still makes sense and this
  causes us no trouble. We will ignore this point henceforth.}  For
example, in the commutative and ribbon graph case the difference
between two types of graph complex lies in a different notion of
orientation: a twisted orientation corresponds to ordering edges of a
graph, as opposed to ordering the half edges and vertices.

More generally, one can replace cyclic operads $\cO$ by arbitrary
(twisted) modular operads. Then, if $\cO$ admits a weakly unital
multiplication, one still obtains in the same manner the result that
the graph complex for $\cO$ is quasi-isomorphic to the subcomplex
defined analogously to the above.

Similarly, we can generalize all the constructions of the preceding
section from the cyclic to the modular setting. Let $\cO$ be an
arbitrary (twisted) modular operad $\cO$.  In this case, one replaces
$C(\cO)$ (used in the cyclic case above) by the Feynman transform
(twisted) modular operad, ${\mathsf F}\mathcal O$, and modular
algebras over this operad are then thought of as curved algebras. The
constructions, results (Claims \ref{opclaim}.(ii) and \ref{opclaim3}),
and proofs carry over to this setting.  In the special case where one
has a cyclic operad and considers it $\Det$-twisted modular using the
naive $\Det$-twisted modular closure, one recovers the above results.
Another example concerns so-called quantum $A_\infty$-algebras: Let
$\mathcal{O}'$ be the $\Det$-twisted modular closure of
$\mathscr{A}\!ss$ such that ${\mathsf F}\mathcal O'$-algebras are
so-called quantum $A_\infty$-algebras (see for instance \cite[Example
5.2]{ChLafdmm} for an explanation of this notion). Then, if we let
$\mathcal{O} = \mathcal{O}' \oplus \CC[0]$ be the unital version
(adding a $0$-ary operation providing a unit for the associative
multiplication), then ${\mathsf F}\mathcal O$ defines a notion of
curved quantum $A_\infty$-algebras.  As in the cyclic $A_\infty$ case,
one sees that nontrivially curved quantum $A_\infty$-algebra
structures on a fixed vector space are gauge equivalent to those where
all the operations of positive arity are of the form
\eqref{opclaim12eq}, and in particular all land in a fixed
one-dimensional vector space. We refrain from making precise
statements.

\subsection{Examples}
Here we provide examples of the preceding constructions for Poisson, Gerstenhaber, BV, and permutation (or pre-Lie) algebras.
\begin{example}\label{pgexam} Consider the case of Poisson algebras.
  It is natural to consider unital Poisson algebras, where here a unit
  $f$ is an element satisfying $\{1, f\} = \{f, 1\} = 0$ for all $f$,
  whereas $1 \cdot f = f \cdot 1 = f$. Let $\mathcal{O}$ be the operad
  governing unital Poisson algebras, i.e., $\mathcal{O} =
  u\mathscr{P}\!\textit{oiss} = \mathscr{P}\!\textit{oiss} \oplus
  \CC[0]$ where $\mathscr{P}\!\textit{oiss}$ is the Poisson operad and
  $\CC[0]$ is the one-dimensional vector space concentrated in degree
  zero, which has zero compositions with the bracket of $\mathcal{O}$,
  and composition $m_2(m_0 \otimes \Id) = \Id = m_2(\Id \otimes m_0)$
  with the commutative multiplication, i.e.,
  $\mathscr{P}\!\textit{oiss} \supset u\mathscr{C}\!\textit{omm}$.
  Then, by the above, one can think of $C(\mathcal{O})$-algebras as
  curved Poisson-infinity algebras (since Poisson is Koszul self-dual,
  as in the associative case), and it follows from Claim
  \ref{opclaim}.(i) that nontrivially curved Poisson-infinity structures
  on a vector space are all gauge equivalent.

  Similarly, the Gerstenhaber operad $\mathscr{G}\!\textit{erst}$ is
  Koszul and dual to its suspension $\mathfrak{s}
  \mathscr{G}\!\textit{erst}$; here the suspension is defined by
  tensoring by the endomorphism operad of the one-dimensional odd
  vector space (as a $\ZZ/2$-graded $\SS$-module, this means that one
  applies $\Pi$ to the even-ary part).  As before, one can consider
  unital $\mathscr{G}\!\textit{erst}$ algebras,
  $u\mathscr{G}\!\textit{erst}$, and similarly its suspension $\mathfrak{s}(u
  \mathscr{G} \! \textit{erst})$. Let $\cO$ be this latter operad,
  which we can also think of as $u
  \mathscr{G}\!\textit{erst}^!$. Here, the unit is odd.  We define
  curved $\mathscr{G}\!\textit{erst}_\infty$ algebras as algebras over
  $C(\cO)$, where now the curvature is odd, rather than even.  From
  Claim \ref{opclaim}.(i), we deduce that any two nontrivially curved
  Gerstenhaber-infinity algebra structures on the same graded vector
  space are gauge equivalent.

  In the case of the Poisson operad (but not for the Gerstenhaber operad),
  in fact $\mathscr{P}\!\textit{oiss}$ is cyclic, and hence one obtains
  the notion of cyclic curved Poisson-infinity algebras. Claim
  \ref{opclaim}.(ii) then implies that all nontrivially curved cyclic
  Poisson-infinity structures on a vector space are gauge equivalent
  to those for which all operations of positive arity are of the form
  \eqref{opclaim12eq}, and in particular all land in a fixed
  one-dimensional vector space.

  Moreover, in fact one can compute the associated (unital Poisson)
  graph homology. By Claim \ref{opclaim2}, the associated graph
  complex is quasi-isomorphic to the subcomplex spanned by graphs
  whose vertices are all univalent except for at most one.  Since the
  Poisson operad is the associated graded operad of the associative
  operad, and in particular has the same $\SS$-module structure, one
  sees that this subcomplex has zero differential, and is spanned by
  the star-shaped graphs with central vertex of odd valence.  That is,
  the graph homology for the unital Poisson operad is isomorphic to
  the homology of the graph complex $\widetilde{\mathcal{G}_r}$ for the unital
  associative operad.
\end{example}
\begin{example}\label{bvexam} The BV operad $\BV$ \cite{Getbva} is the homology operad of the operad of framed little discs; it is known to be cyclic. Its algebras, called $\BV$-algebras, are Gerstenhaber algebras together with an odd operator $\Delta$ which is a differential operator of second order with respect to the commutative multiplication and a derivation of the odd Lie bracket.  Let $\overline{\BV}$, as before, denote the augmentation ideal of the operad $\BV$.
  The $C(\overline{\BV})$-graph complex computes, according to
  \cite{Giafl2d}, the homology of the classifying space of diffeomorphism
  groups of 3-dimensional oriented handlebodies.  Although the operad
  $\BV$ is not defined by (homogeneous) quadratic relations, in
  \cite{GTVhbva} it was shown to be Koszul in a more general sense,
  and its Koszul dual, $\BV^!$ was described, and shown to be
  quasi-isomorphic to $C(\overline{\BV})$.

  A unital $\BV$-algebra is a $\BV$-algebra with a unit with respect
  to the commutative multiplication and such that the value of
  $\Delta$ on the unit is zero.

  Consider the operad ${\mathcal O} =u\BV$ governing unital
  $\BV$-algebras; then we can view $C(\mathcal O)$ as the operad
  governing curved $\BV^!_\infty$-algebras. It follows that all
  nontrivially curved $\BV^!_\infty$-algebras are gauge equivalent.
Furthermore, according to Claim
  \ref{opclaim}.(ii) all nontrivially curved cyclic $\BV^!_\infty$-algebras are gauge equivalent
  to those for which all operations of positive arity are of the form
  \eqref{opclaim12eq}.

  Finally, by Claim \ref{opclaim2}, the graph complex for $\mathcal O$
  is quasi-isomorphic to the subcomplex with at most one vertex of
  valence $\geq 2$.
\end{example}
\begin{example}\label{permexam} Consider the case of (right) pre-Lie
  algebras, i.e., those with a single operation $\star$ satisfying the
  relation
\begin{equation}
x \star (y \star z) - (x \star y) \star z = x \star (z \star y) - (x \star z) \star y.
\end{equation}
For such algebras, it makes perfect sense to define the notion of a unit, $1$, such that
\begin{equation} \label{prelieunit}
1 \star x = x = x \star 1.
\end{equation}
We thus obtain an operad $\mathcal{O} = u \mathscr{P}\!\textit{re-Lie}
= \mathscr{P}\!\textit{re-Lie} \oplus \CC[0]$ governing \emph{unital}
pre-Lie algebras, where now the compositions with $\CC[0]$ are given
by \eqref{prelieunit}, or more precisely, $m_2(m_0 \otimes \Id) = \Id
= m_2(\Id \otimes m_0)$ where $m_0 = 1 \in \CC[0]$ and $m_2 \in
\mathscr{P}\!\textit{re-Lie}[2]$ is the multiplication operation
$\star$.  By the above procedure, we then obtain a type of curved
algebra, namely algebras over $C(\mathcal{O})$.  We will call these
 \emph{curved (right)
  $\mathscr{P}\!\textit{erm}_{\infty}$-algebras}, since the operad
(right) $\mathscr{P}\!\textit{erm}$ is Koszul dual to (right)
$\mathscr{P}\!\textit{re-Lie}$. In particular, curved
  $\mathscr{P}\!\textit{erm}_\infty$ algebras with zero curvature
  ($m_0=0$, i.e., the corresponding derivation of $C_{\mathcal{O}}(V)$
  has zero constant term) are the same as ordinary
  $\mathscr{P}\!\textit{erm}_\infty$ algebras.  Let us recall here
  that ordinary permutation algebras are algebras with an operation
  $\circ$ satisfying the relation
\begin{equation}
x \circ (z \circ y) = x \circ (y \circ z) = (x \circ y) \circ z,
\end{equation}
i.e., associative algebras additionally satisfying the first equality
above.  The operad $\mathscr{P}\!\textit{erm}$ is the one whose algebras are
permutation algebras.

From the above results, we deduce that any two curved
$\mathscr{P}\!\textit{erm}_\infty$-algebra structures on $V$ with nonzero curvature
are gauge-equivalent.

Note that, strictly speaking the operad $\mathscr{P}\!\textit{re-Lie}$ is not cyclic, it is
anticyclic \cite[p. 9]{GKcoch}; however we can ignore this difference; any anticyclic operad gives rise to a cyclic one via the operadic suspension. Therefore, we obtain the corresponding result on the classification of nontrivially curved (anti)cyclic $\mathscr{P}\!\textit{erm}_\infty$-algebras, and on
$u\mathscr{P}\!\textit{re-Lie}$-graph homology.
\end{example}
\subsection{On the cobar construction of unital operads}
Finally, we remark that there is a certain subtlety associated with
taking cobar-constructions of unital operads (such as an operad with a
weakly unital multiplication), as we do. Let $\mathcal O$ be
such an operad. Then it is easy to see that for any $n\geq 0$ the
complex $C{\mathcal O}(n)$ is contractible. However it does not follow
that any $C\mathcal O$-algebra is gauge equivalent to a trivial
one. Indeed, one can take $\mathcal O$ to be the operad $u{\mathscr
  A}\!\textit{ss}$ or $u{\mathscr C}\!\textit{omm}$ governing unital
associative or unital commutative algebras; then $C\mathcal
O$-algebras (on a graded vector space $V$ with zero differential) will
be curved $A_\infty$- or $L_\infty$-algebras, respectively, and there is
no reason for an arbitrary curved $A_\infty$- or $L_\infty$-algebra to
be gauge equivalent to a trivial one.  In fact, even if we let
$\mathcal O$ be the operad ${\mathscr A}\!\textit{ss}$ or ${\mathscr
  C}\!\textit{omm}$, and take the bar construction of it (i.e., of the
whole unital operad, rather than just the augmentation ideal as one
usually does in these cases), we again obtain that $C \mathcal
O$-algebras will be ordinary $A_\infty$- or $L_\infty$-algebras, with
the usual (nontrivial) gauge equivalence relation, even though $C
\mathcal{O}$ remains acyclic.

The explanation of this apparent paradox is that the operad $C\mathcal
O$ is \emph{not cofibrant}; see \cite{BMahto} for this notion. An algebra
over $C\mathcal O$ is a map from $C\mathcal O$ to an endomorphism
operad of a dg vector space and this map is not necessarily homotopic
to zero even though $C\mathcal O$ is acyclic. A similar phenomenon
occurs when considering a cobar-construction $(T\Pi V^*,d)$ for a
unital associative algebra $V$; dg maps from $T\Pi V^*$ to the field
$\CC$ are the Maurer-Cartan elements in $V$, i.e. the odd elements
$v\in V$ for which $dv+v^2=0$. Such elements need not be gauge
equivalent to zero despite $(T\Pi V^*,d)$ being acyclic; again,
precisely because $(T\Pi V^*,d)$ is not a cofibrant dga.

Furthermore, given a cyclic operad $\mathcal O$ as above we can form its $\Det^d$-modular closure $\overline{\mathcal O}^d$ and its naive $\Det^d$-modular closure $\underline{\mathcal O}^d$. Then the same reasoning
shows that the complexes ${\mathsf F}\overline{\mathcal O}^d((n))$  and ${\mathsf F}\underline{\mathcal O}^d((n))$ are acyclic
for $n>0$; here ${\mathsf F}{\mathcal O}$ is the Feynman transform of
$\mathcal O$.

Next, the complexes  ${\mathsf F}\underline{u{\mathscr C}\!\textit{omm}}^1((0))$ and ${\mathsf
  F}\underline{u{\mathscr A}\!\textit{ss}}^1((0))$ are
nothing but our graph complexes $\widetilde{\mathcal G}$ and
$\widetilde{\mathcal G}_r$. For $n > 0$, ${\mathsf F}\underline{u{\mathscr C}\!\textit{omm}}^1((n))$ and ${\mathsf
  F}\underline{u{\mathscr A}\!\textit{ss}}^1((n))$ are similar except they are
spanned by graphs which are additionally equipped with $n$ external
labeled edges (legs) which are not allowed to be contracted. When legs
are present, these graph complexes are therefore acyclic.  However,
the vacuum (legless) part  ${\mathsf F}\overline{\mathcal O}^d((n))$  and ${\mathsf F}\underline{\mathcal O}^d((n))$ of the
Feynman transform need not be acyclic, as our results demonstrate.

\vskip 10 pt
\noindent
\bibliographystyle{amsalpha}
\def\cprime{$'$}
\providecommand{\bysame}{\leavevmode\hbox to3em{\hrulefill}\thinspace}
\providecommand{\MR}{\relax\ifhmode\unskip\space\fi MR }
\providecommand{\MRhref}[2]{%
  \href{http://www.ams.org/mathscinet-getitem?mr=#1}{#2}
}
\providecommand{\href}[2]{#2}

\end{document}